\documentclass[review,hidelinks,onefignum,onetabnum]{siamart250211}

\nolinenumbers

\usepackage{lipsum}
\usepackage{amsfonts}
\usepackage{graphicx}
\usepackage{epstopdf}
\usepackage{algpseudocode}

\usepackage{amssymb,amsfonts,amsmath,latexsym,dsfont}
\usepackage{mathtools}

\usepackage{mathrsfs}
\usepackage{setspace}
\usepackage{subfigure}
\usepackage{fancyhdr}
\usepackage{wrapfig}

\usepackage{booktabs}
\usepackage[normalem]{ulem}

\usepackage{boxedminipage}
\usepackage{pgf,pgfarrows} %% pdf-drawing package
\usepackage{subfigure} %% pdf-drawing package

\newcommand{\FrameboxA}[2][]{#2}
\newcommand{\Framebox}[1][]{\FrameboxA}

\newcommand{\mc}[3]{\multicolumn{#1}{#2}{#3}}

\renewcommand{\div}{\nabla\cdot\,}

\newcommand{\grad}{\nabla}

\newcommand{\im}{\textit{\i}}

\newcommand{\bfe}{{\bf e}}

\newcommand{\bfp}{{\bf p}}
\newcommand{\bfx}{ {\bf x}}

\newcommand{\bfv}{ {\bf v}}

\newcommand{\bfu}{{\bf u}}
\newcommand{\bfq}{{\bf q}}

\newcommand{\bfr}{{\bf r}}
\newcommand{\bfw}{{\bf w}}

\newcommand{\bfgamma}{{\boldsymbol \gamma}}

\newcommand{\bfmu}{{\boldsymbol \mu}}
\newcommand{\bflambda}{{\boldsymbol \lambda}}
\newcommand{\bfrho}{{\boldsymbol \rho}}

\renewcommand{\theequation}{\arabic{section}.\arabic{equation}}

\definecolor{darkblue}{rgb}{0.0, 0.1, 0.8}
\definecolor{darkred}{rgb}{0.6, 0.0, 0.1}
\definecolor{darkpurple}{rgb}{0.5, 0.0, 0.5}
\definecolor{darkgreen}{rgb}{0.0, 0.5, 0.2}

\newcommand{\dedication}[1]{
  \begin{center}
    \emph{#1}
    \vspace{5pt}
  \end{center}
}

\graphicspath{{images/}{./}}

\ifpdf
  \DeclareGraphicsExtensions{.eps,.pdf,.png,.jpg}
\else
  \DeclareGraphicsExtensions{.eps}
\fi

\newsiamremark{remark}{Remark}
\headers{Block-acoustic precond. for the elastic Helmholtz equation}{R. Yovel and E. Treister}

\title{A block-acoustic preconditioner for the elastic Helmholtz equation \thanks{Corresponding author: Rachel Yovel. \funding{This research was supported by The Israel Science Foundation (grant No. 656/23). RY is supported by the Ariane de Rothschild scholarship and by Kreitman High-tech scholarship. The authors also thank the Lynn and William Frankel Center for Computer Science at BGU.}}}

\author{Rachel Yovel\thanks{Ben-Gurion University, Beer Sheva, Israel.
  \email{yovelr@bgu.ac.il, erant@cs.bgu.ac.il}}
\and Eran Treister$^\dag$}

\usepackage{amsopn}

\begin{document}

\maketitle

\dedication{Dedicated to the memory of Howard Elman}

\begin{abstract}
We present a novel block-preconditioner for the elastic Helmholtz equation, based on a reduction to acoustic Helmholtz equations. 
Both versions of the Helmholtz equations  are challenging numerically. 
The elastic Helmholtz equation is larger, as a system of PDEs, and harder to solve due to its more complicated physics.
It was recently suggested that the elastic Helmholtz equation can be reformulated as a generalized saddle-point system, opening the door to the current development.
Utilizing the approximate commutativity of the underlying differential operators, we suggest a block-triangular preconditioner whose diagonal blocks are acoustic Helmholtz operators. 
Thus, we enable the solution of the elastic version using virtually any existing solver for the acoustic version as a black-box. 
We prove a sufficient condition for the convergence of our method, that sheds light on the long questioned role of the commutator in the convergence of approximate commutator preconditioners.
We show scalability of our preconditioner with respect to the Poisson ratio and with respect to the grid size. 
We compare our approach, combined with multigrid solve of each block, to a recent monolithic multigrid method for the elastic Helmholtz equation. 
The block-acoustic multigrid achieves a lower computational cost for various heterogeneous media, and a significantly lower memory consumption, compared to the monolithic approach. 
It results in a fast solution method for wave propagation problems in challenging heterogeneous media in 2D and 3D.

\end{abstract}



\section{Introduction}
The Helmholtz equations model wave propagation in the frequency domain. 
The acoustic Helmholtz equation models acoustics \cite{wang2010acoustic, elec2001acoustic} and electromagnetic waves \cite{zubeldia2012energy}, whereas the elastic Helmholtz equation models waves in solid media, such as earth subsurface \cite{wang20113d}. 
Both versions are difficult for numerical solution, and counted as open problems. 
The resulting linear systems are complex, indefinite, and large, since discretizing high-frequency waves requires very fine meshes. 
The elastic equation amplifies these difficulties over the acoustic one: as a system of equations, it is larger, and the more complicated physics --- including both shear and pressure waves --- gives rise to the need for special care. 

Many solvers and preconditioners were suggested for the acoustic Helmholtz equation, including geometric and algebraic multigrid methods 
\cite{erlangga2006novel, olson2010smoothed, oosterlee2010shifted, livshits2014scalable, tsuji2015augmented, cools2015multi}, domain decomposition methods
\cite{gander2013domain, stolk2013rapidly, tsuji2014sweeping, zepeda2018nested, heikkola2019parallel},
deflation methods
\cite{chen2024matrix},
deep learning methods 
\cite{azulay2022multigrid}, and other methods
\cite{haber2011fast, gordon2013robust, wang2020taylor}.
Yet, just a few solvers are available for the elastic version \cite{rizzuti2016multigrid, baumann2018msss, bonev2022hierarchical}, most of which are multigrid based.
In \cite{treister2024hybrid}, a monolithic multigrid solver is proposed, achieving scalability with respect to the Poisson ratio. 
A main ingredient in the method suggested there is the introduction of a mixed formulation for the elastic Helmholtz equation. 
That is, writing the equation as a generalized saddle-point system with an indefinite leading block. 
This formulation opens the door to specifically tailored block preconditioners for the elastic Helmholtz equation.

There is an abundance of research on block preconditioners for saddle-point systems, especially in the context of incompressible fluid flow. 
For a comprehensive overview, see the review \cite{benzi2005numerical}, Chapter 9 in the book \cite{elman2014finite}, and references therein. 
A main approach is the block-diagonal or block-triangular Schur-complement preconditioners, based on a block elimination of the given system. 
Eigenvalue bounds for Schur-complement preconditioners are given in \cite{murphy2000note, axelsson2006eigenvalue}, extended in \cite{ipsen2001note} to the non-symmetric case. 
However, forming and inverting the Schur-complement is very costly, and broad research in the last decades is dedicated to  the search of cheap approximations for the inverse of the Schur-complement. 

One way to approximate the inverse Schur-complement is based on the notion of approximate commutators \cite{elman1997perturbation, silvester2001efficient, elman2006block, elman2008least, elman2009boundary,  pearson2015development}, that was suggested in the context of incompressible fluid flow.
The $F_p$, or pressure convection-diffusion (PCD) preconditioner \cite{elman2005preconditioning}, creates an approximation of the leading block by re-discretizing it on the pressure space.
The BFBt, or least squares commutator (LSC) preconditioner \cite{elman2005preconditioning, elman2008least}, is based on a more algebraic observation: 
seeking an approximation for the leading block in the pressure's space that will minimize the commutator in a least squares sense. 
Weighted adaptations of the BFBt preconditioner directed to Stokes-like systems with variable viscosity were suggested in \cite{rudi2015extreme, rudi2017weighted}. 
Some Schur-complement-free preconditioners were also suggested, directly utilizing the commutation relations to prevent the inversion of the leading block, see the pressure-Poisson approach  \cite{shirokoff2011efficient, johnston2004accurate}. 

However, most of the research in this direction is dedicated to incompressible fluid flow problems, which are different in nature from the elastic Helmholtz equation.
First, the elastic Helmholtz equation typically deals with compressible materials.
Algebraically speaking, the resulting saddle-point matrix for the elastic Helmholtz equation has a nonzero regularization block, regardless of the discretization.
Second, the leading block is indefinite.
There is some literature on the non-SPD case, e.g, the well-known augmented Lagrangian approach \cite{golub2003solving, benzi2006augmented}, that deals with the case of a singular leading block.
To the best of our knowledge, no approximate commutator preconditioners were suggested for generalized saddle-point systems with an indefinite leading block, prior to this work. 

In this work, we present a Schur-complement-free approximate commutator preconditioner.
Our block-acoustic preconditioner is block-triangular, and its diagonal blocks are acoustic Helmholtz operators. 
Reducing a large, coupled problem into smaller decoupled problems requires less memory access in residual computations. By this, the method introduces natural parallelization.
Moreover, this reduction-based framework enables the use of any advanced solver for acoustic Helmholtz within the solution of the elastic version. 
Our preconditioner scales well with respect to the Poisson ratio, and given a direct solve of each block, is also scalable with respect to the grid size. 
We combine our preconditioner with a multigrid solve of each block, and show that it outperforms monolithic multigrid, for real world geophysical problems in 2D and 3D.

The paper is organized as follows: in Section \ref{sec:background} we give mathematical background, including multigrid methods and approximate commutator preconditioners.
In Section \ref{sec:derivation} we derive the method and prove a theoretical result, shedding light on the role of the commutator in our method.
In Section \ref{sec:results}
we test our method numerically, and finally, in Section \ref{sec:conclusion} we give concluding remarks and discuss future work.

\section{Background}\label{sec:background} 
In this section we present the acoustic and elastic Helmholtz equations. Furthermore, we give some general background on multigrid and specifically on the multigrid methods for the Helmholtz equations that we use in this work. Finally, we give some background on approximate commutator preconditioning.

\subsection{The acoustic and elastic Helmholtz equation}\label{subsec:helm_background}
Let $p = p(\vec{x}), \, \vec{x}\in\Omega$ be the Fourier transform of the wave's pressure field, let $\omega = 2\pi f$ be the angular frequency, $\kappa = \kappa(\vec{x}) > 0$ the ``slowness'' of the wave in the medium (the inverse of the wave velocity).
Let $\rho = \rho(\vec{x}) > 0$ be the density of the medium and $q(\vec{x})$  the source of the waves. Then the acoustic Helmholtz equation is given by 
\begin{equation} \label{eq:acousitcHelm}
\rho \div \left(\rho^{-1}\grad p\right) + \omega^{2}\kappa^2\left(1-\frac{\gamma}{\omega}\im\right)p = q
\end{equation}
where $\im$ stands for the imaginary unit and $\gamma$ represents the physical attenuation.
To solve this equation numerically, we discretize it in a finite domain by a finite differences scheme. 
It is usually equipped with absorbing boundary conditions (ABC) \cite{engquist1977absorbing} or perfectly matched layers (PML) \cite{berenger1994perfectly, singer2004perfectly, harari2000analytical, rabinovich2010comparison}, to mimic the propagation of a wave in an open domain, and to avoid reflections from the boundary. 

The elastic Helmholtz equation in an isotropic medium is given by
\begin{equation}\label{eq:elasticHelm}
\grad\lambda\div\vec u + \vec\div\mu\left(\vec\grad\vec u+\vec\grad\vec u^T\right) +\rho \omega^2\left(1-\frac{\gamma}{\omega}\im\right) \vec u = \vec q_{s}
\end{equation}
where $\vec u = \vec u(\vec x)$ is the displacement vector,
or by the following formulation\footnote{The two formulations are equivalent in homogeneous media, and resembles each other well for heterogeneous media.}
\begin{equation}\label{eq:elasticHelmEquivalent}
\grad(\lambda + \mu)\div\vec u + \vec\div\mu\vec\grad\vec u +\rho \omega^2\left(1-\frac{\gamma}{\omega}\im\right) \vec u = \vec q_{s},
\end{equation}
where $\lambda=\lambda(\vec x)$ and $\mu=\mu(\vec x)$ are the Lam{\'e} coefficients, that represents the stress-strain relationship in the material. These coefficient determine the pressure wave velocity $V_p = \sqrt{(\lambda+2\mu)/\rho}$ and the shear wave velocity $V_s = \sqrt{\mu/\rho}$, see e.g. \cite{li2016optimal}. The Poisson ratio is given by $\sigma = \lambda/2(\lambda+\mu)$, and describes deformation properties of the material. For most materials, $0<\sigma<0.5$, and in the nearly incompressible case (that is counted as the most difficult case), $\sigma\to 0.5$ or $\lambda \gg \mu$. 

In the incompressible case, the elastic Helmholtz equation is reducible to the acoustic Helmholtz equation: substituting $\mu=0$ in \eqref{eq:elasticHelmEquivalent} and applying $\rho\div\rho^{-1}$ on both sides, yields
\begin{equation}\label{eq:mu0_reduction}
\rho\div\rho^{-1}\grad\lambda (\div\vec u)+\frac{\rho\omega^{2}}{\lambda}\left(1-\frac{\gamma}{\omega}\im\right)\lambda\div\vec u = \rho\div\rho^{-1}\vec q_{s}.
\end{equation}
That is, an acoustic Helmholtz equation, similar to \eqref{eq:acousitcHelm}, with $\lambda\div\vec u$ as the unknown scalar function. Note that the wave velocity of \eqref{eq:mu0_reduction} is equal to the pressure wave velocity of the original equation, \eqref{eq:elasticHelmEquivalent}.
Such a reduction is the core idea behind our method.
However, the reduction above is possible (and exact), only when $\mu=0$, the material is incompressible and there are no shear waves. A similar reduction in the compressible case is less trivial, as it involves both shear and pressure wave velocities. 

\begin{figure}
  \centering
  \includegraphics[width=0.35\textwidth]{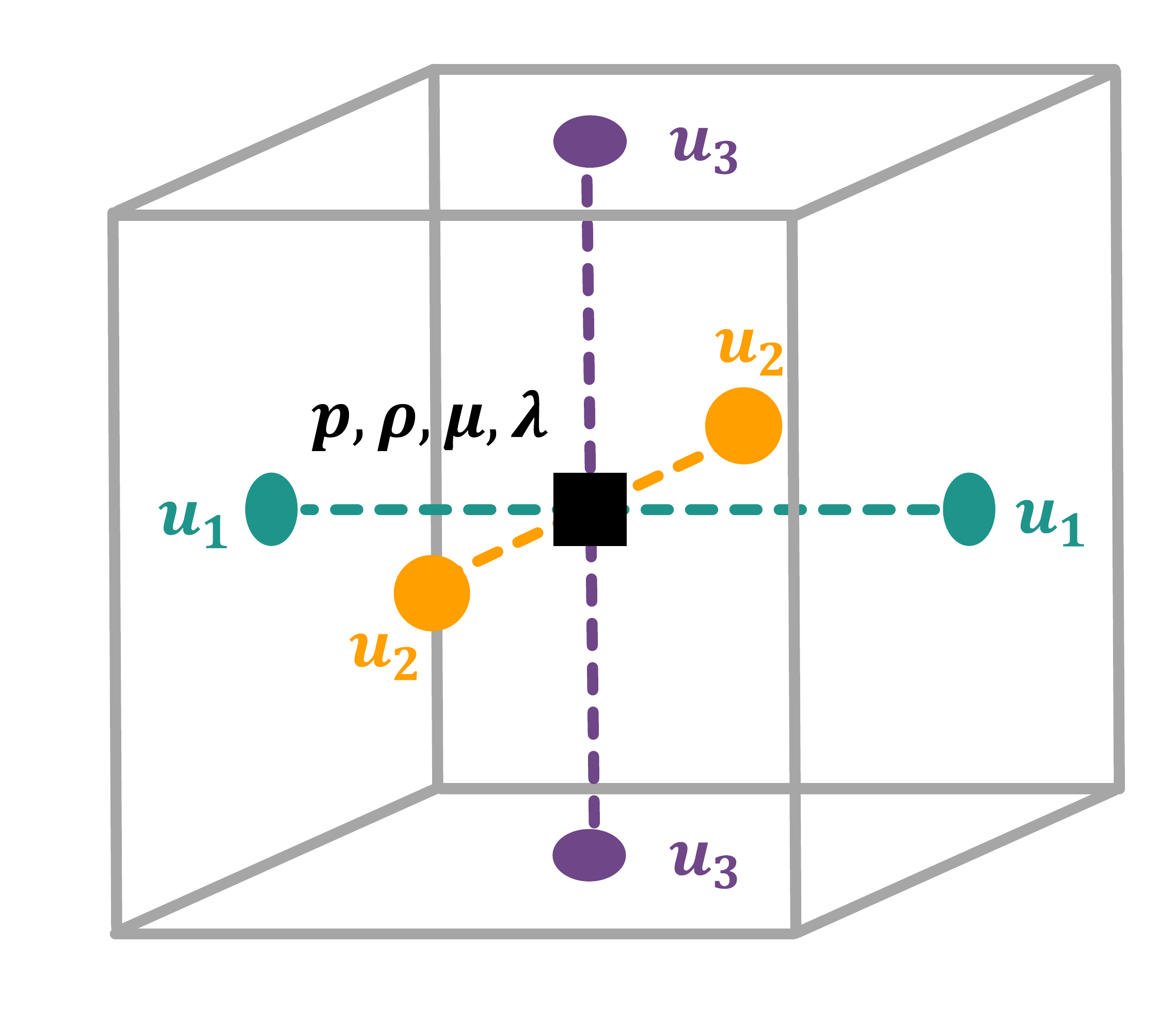}\\
  \caption{A MAC discretization cell in 3D.}\label{fig:cell}
\end{figure}

The mixed formulation, originally suggested for linear elasticity problems \cite{gaspar2008distributive,zhu2010efficient}   and recently suggested for the elastic Helmholtz equation \cite{treister2024hybrid}, is derived as follows. By introducing a new pressure variable $p = -(\lambda + \mu)\div\vec{u}$ and substituting it in \eqref{eq:elasticHelmEquivalent},  we have
\begin{equation}\label{eq:mixed-continuous}
\begin{pmatrix}
-\vec\grad\cdot \mu \vec\grad - \rho \omega^2 \left(1-\frac{\gamma}{\omega}\im\right) & \grad \\
-\div & -\frac{1}{\lambda+\mu}
\end{pmatrix}
\begin{pmatrix}
\vec{u} \\
p
\end{pmatrix}
= \begin{pmatrix}
-\vec{q}_s \\
0
\end{pmatrix}.
\end{equation}
We discretize \eqref{eq:mixed-continuous} using a marker and cell (MAC)  discretization \cite{harlow1965numerical} with standard second-order finite difference stencil. We locate the displacement components 
$\bfu_1,\bfu_2$ and $\bfu_3$ 
on different faces, and the pressure 
$\bfp$ 
as well as the physical variables 
$\bfgamma,\bfrho,\bflambda,\bfmu$ 
in the cell centers, see Figure \ref{fig:cell}.
The discretized system is given by
\begin{equation}\label{eq:mixed-discretized}
\underbrace{
\begin{pmatrix}
A & B^T \\
B & -C
\end{pmatrix}
}_{\mathcal{H}}
\begin{pmatrix}
\vec\bfu\\
\bfp \end{pmatrix}
\coloneqq
\begin{pmatrix} \vec{\grad}_h^T A_e(\bfmu)\vec{\grad}_h - \omega^{2}M & \grad_h \\ \grad^T_h & \mbox{diag}\left(-\frac{1}{\bflambda + \bfmu}\right) \end{pmatrix}
\begin{pmatrix}
\vec\bfu\\
\bfp \end{pmatrix} = \begin{pmatrix}
-\vec\bfq\\
0 \end{pmatrix}
\end{equation}
where $A_e(\cdot)$ is an edge-averaging operator and
$M=A_f(\bfrho(1-(\bfgamma/\omega) \im)$ is a mass matrix, with $A_f(\cdot)$ being a face-averaging operator.
A key feature of the resulting matrix, is that its main block is a block-diagonal matrix, comprised of three acoustic Helmholtz operators (in 3D). 
We utilize this block structure, depicted in Figure \ref{fig:blocks_mixed_3d}, in our preconditioning approach.
I.e., the block-acoustic preconditioner that we present in the following section is a block-triangular matrix with acoustic Helmholtz operators on its diagonal, see Fig. \ref{fig:blocks_prec_3d}.

\begin{figure}
\begin{center}
	\newcommand{\image}[1]{\includegraphics[width=0.25\linewidth]{#1}}
    \subfigure[\footnotesize Operator]{\image{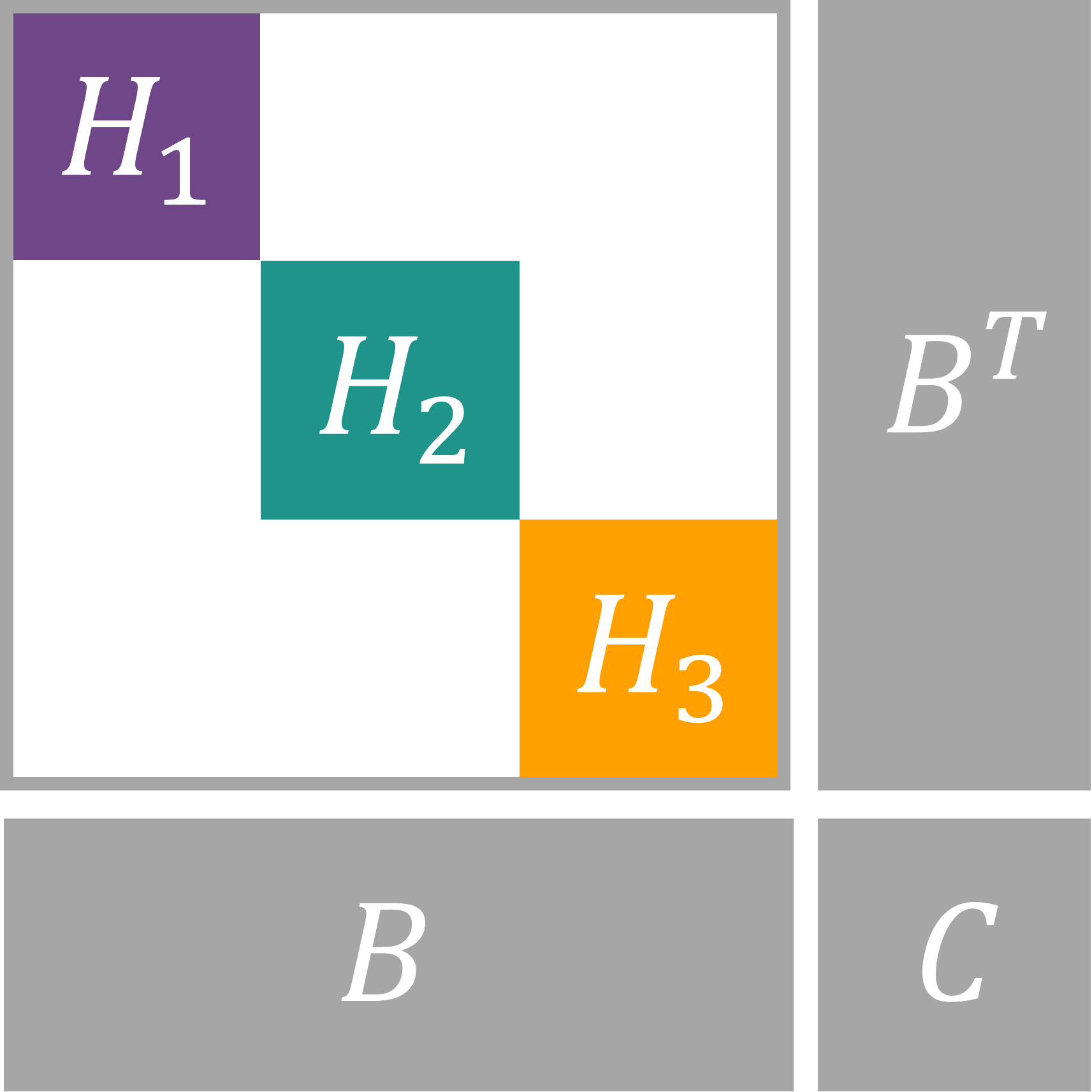}\label{fig:blocks_mixed_3d}}
    \hspace{50pt}
    \subfigure[\footnotesize Preconditioner]{\image{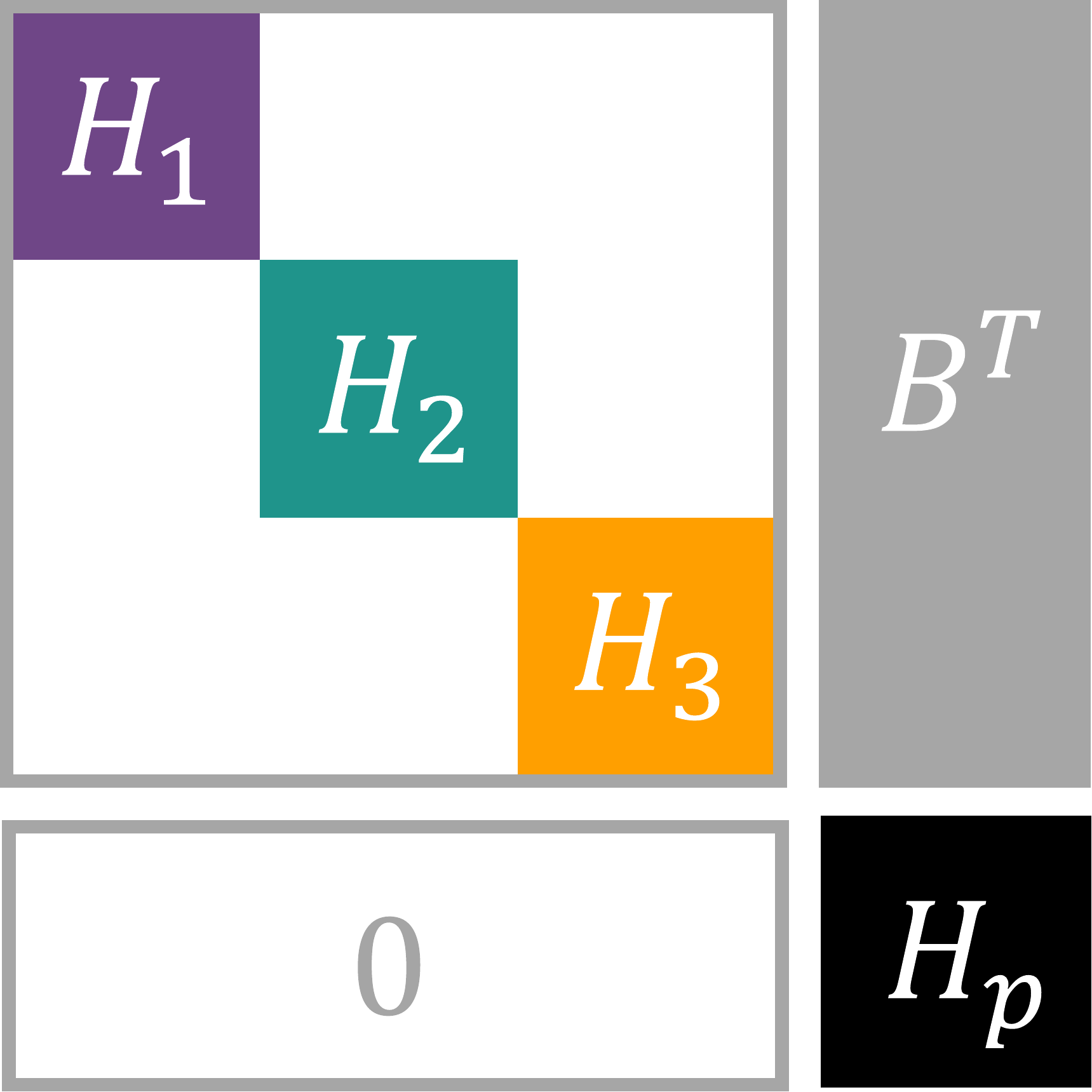}\label{fig:blocks_prec_3d}}\\
\end{center}
\caption{On the left, the block structure of the discretized elastic Helmholtz equation in mixed formulation in 3D. 
On the right, the block structure of the block-acoustic preconditioner in 3D.
}\label{fig:Blocks}
\end{figure}

\subsection{Shifted Laplacian multigrid}\label{subsec:sl}

The complex shifted Laplacian multigrid preconditioner (CSLP) \cite{erlangga2006novel} is a well known approach for the solution of the acoustic Helmholtz equation. 
Recently, an efficient adaptation for this method have been developed for the elastic Helmholtz equation  \cite{treister2024hybrid}.
Before presenting the two methods, we give general background on multigrid. 

Multigrid methods \cite{brandt1977multi} are a family of iterative solvers for linear systems of the form $H\bfu=\bfq$,
that emerges from discretizations of PDEs.
These methods are based on two complementary process: 
\emph{smoothing}, or relaxation, to annihilate the high frequency modes of the error, and \emph{coarse grid correction} to annihilate the remaining low frequency error modes. 
The former is done by applying an iterative method with smoothing properties, and the latter is done by estimating and correcting the error $\bfe$, typically by solving a coarser analogue of the problem. 
The translation between the coarse and fine grids is done by intergrid operators called \emph{restriction} $R$ and \emph{prolongation} $P$. 
Algorithm \ref{alg:TwoCycle} summarizes the process using two grids. 
By treating the coarse problem recursively with one recursive call, we obtain the multigrid V-cycle, and by treating the coarse problem recursively with two recursive calls we obtain a W-cycle. 
For a more detailed description, see \cite{briggs2000multigrid, trottenberg2000multigrid}.

\iffalse
\begin{algorithm}
\DontPrintSemicolon
\KwSty{Algorithm: $\bfu\leftarrow TwoGrid(H,\bfq,\bfu).$}\;
\begin{enumerate}\Indm
\item Apply pre-relaxations: $\bfu \leftarrow Relax(H,\bfu,\bfq)$\;
\item Compute and restrict the residual $\bfr_c = P^T(\bfq - H\bfu)$.
\item Compute $\bfe_c$ as the solution of the coarse-grid problem $H_c\bfe_c=\bfr_c$.
\item Apply coarse grid correction: $\bfu \leftarrow \bfu + P\bfe_c$.
\item Apply post-relaxations: $\bfu \leftarrow Relax(H,\bfu,\bfq)$.
\end{enumerate}
\caption{Two-grid cycle.}
\label{alg:TwoCycle}
\end{algorithm}
\fi

\begin{algorithm}
\caption{Two-grid cycle}\label{alg:TwoCycle}
\begin{algorithmic}[1]
\Statex \textbf{Algorithm:} $\bfu\leftarrow TwoGrid(H,\bfq,\bfu).$
\State Apply pre-relaxations: $\mathbf{u} \gets \text{Relax}(H,\mathbf{u},\mathbf{q})$
\State Compute and restrict the residual: $\mathbf{r}_c \gets P^\top(\mathbf{q} - H\mathbf{u})$
\State Compute $\mathbf{e}_c$ as the solution of the coarse-grid problem $H_c\mathbf{e}_c = \mathbf{r}_c$
\State Apply coarse-grid correction: $\mathbf{u} \gets \mathbf{u} + P\mathbf{e}_c$
\State Apply post-relaxations: $\mathbf{u} \gets \text{Relax}(H,\mathbf{u},\mathbf{q})$
\end{algorithmic}
\end{algorithm}

Standard multigrid methods are not effective in the solution of the acoustic Helmholtz equation~\eqref{eq:acousitcHelm}. 
As observed in \cite{elman2001multigrid}, the error amplification factor of a coarse grid correction, for a given eigenvector is given by
\begin{equation}\label{eq:1-lam}
1-\frac{\lambda_h}{\lambda_H}
\end{equation}
when $\lambda_h$ is the eigenvalue corresponding to a given eigenvector on the fine grid, and $\lambda_H$ is the eigenvalue corresponding to its coarsening. 
In indefinite systems, when exposed to many near-zero eigenvalues of different signs, the coarsened version of the same vector can lead to an eigenvalue of opposite sign, hence causing divergence of the method. 
An additional complex shift can make the expression in \eqref{eq:1-lam} smaller than one, and hence promise convergence of the multigrid method.
The CSLP approach \cite{erlangga2006novel} is based on this observation. Let $H$ be a matrix defined by a discretization of the Helmholtz operator. Define an attenuated operator
\begin{equation}\label{eq:shift}
H_s = H - \im\alpha\omega^2M_s,
\end{equation}
where $M_s$ is some mass matrix, and $\alpha>0$ is a shifting parameter. 
The shifted version can be solved by multigrid, and serves as a preconditioner for a discretized acoustic Helmholtz equation inside a suitable Krylov method such as (flexible) GMRES \cite{saad1993flexible} or BiCGSTAB \cite{van1992bi}.

For the elastic Helmholtz equation, however, the shifted Laplacian preconditioner performs poorly, without additional adaptations. 
It was suggested in \cite{treister2024hybrid}, to apply a zero-padded shift only on the leading block of the mixed formulation \eqref{eq:mixed-continuous}, rather than shifting the original formualtion \eqref{eq:elasticHelmEquivalent}. 
Together with Vanka relaxation as a smoother, it resulted in an efficient monolithic multigrid preconditioner for the elastic Helmholtz equation. 
This preconditioner scales well with respect to the Poisson ratio, unlike previous multigrid methods for elastic Helmholtz. Yet, it does not scale with respect to the grid size, as the added shift depends on the frequency. 
This behavior is also evident in CSLP for the acoustic equation.

One of the major difficulties in the research of Helmholtz equation (acoustic or elastic) is achieving a wavenumber independent convergence, or, a scalability  with respect to the grid size. 
High-frequency waves require fine meshes, at least 10 grid points per wavelength for standard second-order discretizations. To keep a constant ratio of grid points per wavelength,  the frequency $\omega$ increases in proportion to the grid size. It results in a larger shift in \eqref{eq:shift}, for larger grids. Hence, one cannot expect a grid resolution independence whenever using a shift.

\subsection{Approximate commutator preconditioners} \label{subsec:approx_com_background}

An abundance of applications give rise to saddle-point systems, such as computational fluid dynamics, constrained optimization, finance, optimal control, and discretization of coupled PDE's, to name a few.
Schur-complement preconditioning is a popular approach for saddle-point systems, see e.g. \cite{elman2005preconditioning, silvester2001efficient}. A generalized saddle-point system and its block-triangular approximate Schur-complement preconditioner are typically given by
\begin{equation}\label{eq:general_saddle_and_prec}
\mathcal{A} = 
\begin{pmatrix}
A & B^T \\
B & -C
\end{pmatrix},
\qquad
\mathcal{P} = 
\begin{pmatrix}
A & B^T \\
0 & \tilde{S}
\end{pmatrix}
\end{equation}
where $\tilde{S}$ is an approximation of the Schur-complement $S = C + BA^{-1}B^T$.
Forming and inverting the exact Schur-complement is \emph{very costly} and during the past decades, there is an extensive search for cheap, yet reliable, approximations for its inverse.

Approximate commutator preconditioners form an important family of approximate Schur-complement preconditioners. 
The idea of approximate commutator preconditioning is based on utilizing the commutativity of the following differential operators in the continuous world\footnote{Note that for MAC discretization with constant coefficients and periodic boundary conditions, the commutation is exact in the discrete world, as observed in \cite{brandt1979multigrid}.}
\begin{equation}\label{eq:vecId}
\div \vec\Delta = \Delta\div,
\end{equation}
to construct easy-to-invert approximations for the Schur-complement.
In typical applications, $B$ from Eq. \eqref{eq:general_saddle_and_prec} represents a discrete minus divergence operator, and $A$ is a discrete minus vector Laplacian, with or without an added mass term or convective term. 
The continuous commutation relations \eqref{eq:vecId} does not take into account the locations of the variables in a staggered discretization.
To utilize similar commutation relations in the discrete space, one might look for an operator $A_p$, that mimics $A$ but lives in the pressure's space, such that the commutator
\begin{equation}\label{eq:commutator}
\Xi = BA - A_p B
\end{equation}
is small, in some sense.

Based on this notion, the $F_p$ \cite{elman2005preconditioning} or PCD \cite{elman2009boundary} preconditioner was suggested, for standard Stokes-like systems with $C=0$. In this approach, the Schur-complement is approximated by
\begin{equation}\label{eq:fp}
BA^{-1}B^T \approx B D_u B^T A_p^{-1} D_p \quad \text{or} \quad BA^{-1}B^T \approx D_p A_p^{-1} B D_u B^T
\end{equation}
where $D_u$ and $D_p$ are the lumped velocity and pressure mass matrices resulting from the finite element discretization, and $A_p$ is a re-discretization of $A$ on the pressure's space. 
Another well-studied approximate commutator approach for Stokes-like systems is the BFBt  \cite{elman2005preconditioning, elman2009boundary} or LSC \cite{elman2008least} preconditioning approach, giving the following approximation for the Schur-complement:
\begin{equation}\label{eq:wbfbt}
BA^{-1}B^T \approx (BD_u^{-1}B^T)(BD_u^{-1}AD_u^{-1}B^T)^{-1}(BD_u^{-1}B^T).
\end{equation}
Many works have been done on BFBt-like preconditioners \cite{elman2006block, elman2008taxonomy, kay2002preconditioner, silvester2001efficient}. Some where focused on the adjustment of the method to Stokes problems with variable viscosity. The weighted BFBt methods given in \cite{rudi2015extreme, rudi2017weighted}, for instance, suggest replacing $D_u$ in \eqref{eq:wbfbt} by $\text{diag}(A)$ or by a scaling matrix with the square root of the the viscosity values, to deal with Stokes problems of variable viscosity.

\section{Derivation and analysis of the preconditioner}\label{sec:derivation}

The recent introduction of a mixed-formulation for the elastic Helmholtz equation \eqref{eq:mixed-discretized} raises the need for new preconditioners. 
The resulting saddle-point system has a highly indefinite leading block, non-zero $C$-block and  varying coefficients. 
To this end, we present an approximate commutator Schur-complement-free preconditioner, whose blocks are acoustic Helmholtz operators with shear and pressure wave velocities.

Inspired by the notion of distributive relaxation \cite{brandt1979multigrid, hackbusch2013multi, elman2014finite}, we present our preconditioner as an approximation of a distributed operator.
Applying a left distributor to the system \eqref{eq:mixed-discretized} gives
\begin{equation}\label{eq:mixed-distributor}
\begin{pmatrix}
I & 0 \\
B & -A_p
\end{pmatrix}
\begin{pmatrix}
A & B^T \\
B & -C
\end{pmatrix}
\begin{pmatrix}
\vec\bfe_{\vec{u}}\\
\bfe_{p} \end{pmatrix} = \begin{pmatrix}
I & 0 \\
B & -A_p
\end{pmatrix}\begin{pmatrix}
\vec\bfr_{\vec{u}}\\
\bfr_p \end{pmatrix},
\end{equation}
where $A_p$ should be taken as an acoustic Helmholtz operator with \emph{shear} wave velocity (to mimic $A$), discretized in the pressure's locations. 
We take
\begin{equation}\label{eq:Ap}
A_p \coloneqq BB^T\bfmu - \omega^2 M_p
\end{equation}
where $M_p = \text{diag}\left(\bfrho(1-(\bfgamma/\omega) \im)\right).$
The distributed system is
\begin{equation}\label{eq:mixed-distributed}
\underbrace{
\begin{pmatrix}
A & B^T \\
\Xi & H_p
\end{pmatrix}}
_{\mathcal{K}}
\begin{pmatrix}
\vec\bfe_{\vec{u}}\\
\bfe_{p} \end{pmatrix} 
= 
\begin{pmatrix}
\vec\bfr_{\vec{u}}\\
B\vec\bfr_{\vec u} - A_p\bfr_p 
\end{pmatrix}
\end{equation}
where $\Xi$ is the commutator, as in \eqref{eq:commutator}, and
\begin{equation} \label{eq:Hp}
H_p \coloneqq BB^T + A_p C
\end{equation}
is an acoustic Helmholtz operator with \emph{pressure} wave velocity, discretized on the pressure's space.
Our block-acoustic preconditioner is then given by
\begin{equation}\label{eq:prec}
\mathcal{P} = \begin{pmatrix}
A & B^T \\
0 & H_p
\end{pmatrix}.
\end{equation}
Note that $A$ is a block-diagonal matrix with acoustic Helmholtz operators of shear wave velocity on its diagonal.
Recall that the pressure wave velocity is typically lower than shear wave velocity, making $H_p$ is easier to solve iteratively compared to the diagonal blocks comprising $A$.
The block structure of $\mathcal{P}$, comprised of acoustic Helmholtz diagonal blocks discretized in different locations, is depicted in Fig. \ref{fig:blocks_prec_3d}.
This structure enables the solution of three (in 2D) or four (in 3D) acoustic problems instead of one elastic Helmholtz problem.

\begin{remark}[Choice of $A_p$]
In \eqref{eq:Ap}, we chose to put the weighting $\bfmu$ to multiply the Laplacian from the right. 
One might consider taking $\tilde{A}_p = BA_f(\bfmu)B^T - \omega^2 M_p$ instead of $A_p$, placing the weighting between the minus divergence and the gradient, which better resembles the blocks of $A$, according to \eqref{eq:mixed-discretized}.
Nevertheless, our aim is to reduce the commutator $\Xi$ in \eqref{eq:commutator}.
Specifically, using $A_p$, the dominant part of our commutator (neglecting mass terms) ends up being
\begin{equation}\label{eq:commutator_no_mass}
B((\div)_h A_e(\bfmu) \grad_h) - B B^T \bfmu B,
\end{equation}
while using $\tilde{A}_p$, the corresponding part of the commutator would have been
\begin{equation}\label{eq:commutator_no_mass_tilde}
B((\div)_h A_e(\bfmu) \grad_h) - B A_f(\bfmu)B^T B.
\end{equation}
In \eqref{eq:commutator_no_mass}, a derivative operator is applied to $\bfmu$ twice from the left in both summands.
On the other hand, in \eqref{eq:commutator_no_mass_tilde}, the right term has only one derivative applied to $\bfmu$ from the left.
For smooth media, $\tilde{A}_p$ and $A_p$ yields equivalent methods. Yet, we see in practice a significant gain for choosing $A_p$ for non-smooth media.
\end{remark}

\subsection{A sufficient condition for convergence} \label{subsec:thm}

The commutator has a crucial role in the designing of approximate commutator preconditioners.
However, as Elman, Silvester and Wathen noticed in their book \cite{elman2014finite}, Remark 9.5, approximate commutator preconditioners can sometimes be effective even when the commutator is not small in norm. 
We recall that there, the discussed problems are related to incompressible fluid flow and have a positive definite leading block.
In the context of high frequency elastic Helmholtz problems, we observe the opposite phenomenon: 
even when the commutator is small in norm, the method might diverge.
The theorem below explains why the commutator alone does not suffice to predict convergence of our preconditioner, and suggests a more reliable measure.

\begin{theorem}\label{thm:thm}
Let $\mathcal{K}$ be the distributed matrix from \eqref{eq:mixed-distributed} and let $\mathcal{P}$ be the preconditioner from \eqref{eq:prec}. Denote by $n$ the size of $A$ and by $m$ the size of $A_p$ from \eqref{eq:Ap}.
Then the preconditioned matrix $\mathcal{P}^{-1} \mathcal{K}$ has an eigenvalue $\lambda=1$ with multiplicity of at least $n$.
Namely, the corresponding error iteration matrix 
$T = I-\mathcal{P}^{-1}\mathcal{K}$
has a nullspace of dimension at least $n$.

Moreover, Let $Z$ be the $m\times m$ matrix
\begin{equation}\label{eq:Z}
Z = \Xi A^{-1}B^T H_p^{-1}.
\end{equation}
Then 
\begin{equation}
\text{\upshape{spec}}(T) = \text{\upshape{spec}}(Z) \cup \{0\},
\end{equation}
and consequently, $\rho(T)=\rho(Z).$
\end{theorem}

\begin{proof}
We first prove that the multiplicity of the eigenvalue $1$ in the preconditioned system is at least $n$.
A straightforward calculation shows
\begin{equation}\label{eq:precond_system}
\mathcal{P}^{-1}\mathcal{K} = \begin{pmatrix}
A^{-1} & -A^{-1}B^T H_p^{-1} \\
0 & H_p^{-1}
\end{pmatrix}
\begin{pmatrix}
A & B^T \\
\Xi & H_p
\end{pmatrix}
=
\begin{pmatrix}
I_n - A^{-1}B^T H_p^{-1}\Xi & 0 \\
H_p^{-1}\Xi & I_m
\end{pmatrix}.
\end{equation}
Therefore, it is readily seen that $\mathcal{P}^{-1}\mathcal{K}$ has an eigenvalue $1$ with multiplicity of at least $m$. Furthermore, the $n\times n$ matrix
\begin{equation} \label{eq:Y}
Y = A^{-1}B^T H_p^{-1}\Xi
\end{equation}
is rank-deficient, with 
\begin{equation}
\text{rank}(Y)\leq\text{rank}(B^T) \leq m.
\end{equation}
Hence, $Y$ has a nullspace of at least $n-m$ dimensions, and consequently, $I-Y$ has the eigenvalue $1$ with a multiplicity of at least $n-m$. Overall, the multiplicity of the eigenvalue $1$ in $\mathcal{P}^{-1}\mathcal{K}$ is at least $n$.

Second, we prove that 
\begin{equation}
\text{spec}(T) = \text{spec}(Y).
\end{equation}
A straightforward calculation gives
\begin{equation}\label{eq:error_iter_mat}
T = I - \mathcal{P}^{-1}\mathcal{K} =
\begin{pmatrix}
Y & 0 \\
-H_p^{-1}\Xi & 0
\end{pmatrix}.
\end{equation}
Evidently, the eigenvalues of $T$ are those of $Y$ and zero. Since $Y$ is a rank deficient matrix, it is therefore clear that the spectrums of $Y$ and $T$ are equal.

Finally, we prove that
\begin{equation}
\text{spec}(Y)\setminus \{0\} = \text{spec}(Z) \setminus \{0\}.
\end{equation}
Indeed, let $(\lambda,\bfv)$ be an eigenpair of $Y$ with $\lambda\neq 0$. Denote $\bfw=\Xi\bfv$. Then,
\begin{equation}\label{eq:Yv}
Y\bfv = A^{-1}B^TH_p^{-1}\bfw = \lambda \bfv 
\quad \Rightarrow \quad
\Xi Y\bfv = Z\bfw = \lambda \bfw.
\end{equation}
Notice that $\bfw$ is a nonzero vector (otherwise $\lambda$ would be zero).
Hence, $(\lambda,\bfw)$ is an eigenpair of $Z$, and thus
$
%\label{eq:specYsubseteq_specZ}
\text{spec}(Y) \setminus \{0\} \subseteq \text{spec}(Z) \setminus \{0\}.
$
The inclusion in the other direction $\supseteq$ is achieved similarly, by multiplying $Z\bfv = \lambda\bfv$ from the left by $A^{-1}B^TH_p^{-1}$. 
Consequently, the spectrum of the error iteration matrix is $\text{spec}(Z)\cup \{0\}.$
\end{proof}

As an immediate consequence of Theorem \ref{thm:thm}, a fixed-point iteration using the preconditioner (without a Krylov method),
$\bfx^{(k+1)} = \bfx^{(k)} + \mathcal{P}^{-1}\bfr^{(k)}$,
converges if and only if $\rho(Z)<1$. Moreover, $\rho(Z)$ gives an upper bound for the asymptotic convergence rate of such a method.
Clearly, the convergence of the fixed-point iteration implies the convergence of the corresponding preconditioned Krylov method, and hence $\rho(Z)<1$ is a sufficient condition for convergence of our method.
Yet, the preconditioned Krylov method might converge rapidly even when $\rho(Z)>1$ and the rate also depends on the scattering of the eigenvalues.
To simplify the discussion below, we refer to $\rho(Z)$ as a measure for convergence, even though our method is implemented as a preconditioned Krylov iteration.

Theorem \ref{thm:thm} sheds light on the role of the commutator in determining the convergence:
it suggests the spectrum of $Z$ from Eq. \eqref{eq:Z} as a better measure for convergence.
The matrix $Z$ manages to encapsulate the influence of the frequency and attenuation, which the commutator fails to detect.
Even a small commutator --- multiplied by the inverse of $A$ and of $H_p$ --- might lead to a significantly larger $\rho(Z)$, for high frequency and low attenuation problems.

In fact, we observe that the frequency and the attenuation have a little to no effect on the commutator: two problems with the same heterogeneity model but different frequency and attenuation, have almost the same commutator norm.
This intriguing phenomenon can be explained by splitting the commutator to a Laplacian-related (high-order derivatives) part and a mass-related (low-order derivatives) part:
\begin{equation}
\Xi = \underbrace{B(\vec\div A_e(\bfmu) \vec\grad) - (BB^T\bfmu)B}_{\Xi_{lap}}
- \underbrace{\omega^2(1-(\gamma/\omega)\im) BA_f(\bfrho) - \text{diag}(\bfrho)B}_{\Xi_{mass}}.
\end{equation}
Typically we observe that $\Xi_{mass}\ll\Xi_{lap}$, because we have only first derivatives in $\Xi_{mass}$ (represented by $B$) compared to higher-order derivatives in $\Xi_{lap}$. 
Evidently, the commutator is not sensitive to changes in frequency and attenuation that comprises the mass term.

Nevertheless, the convergence \emph{is} sensitive to the changes in the mass term, and in some cases it requires special treatment.
The convergence deteriorates as the frequency grows and improves as the attenuation grows.
For smooth media, the commutator is typically small enough to overcome near zero eigenvalues in $A$ and $H_p$. 
However, in highly heterogeneous high frequency problems, an additional shift, physically representing an added artificial attenuation, might be necessary for the block-acoustic preconditioner to converge (similarly to its necessity in the context of shifted Laplacian multigrid). 
The additional attenuation shifts the eigenvalues of $A$ and in $H_p$ in the complex plane, hence preventing near zero eigenvalues.
Nevertheless, we still want to solve the original problem, and the shifted problem acts only as an aid, so a large shift can hamper the convergence and interfere with the scalability.

To summarize, although it was long ago stated that the commutator does not determine the convergence, to the best of our knowledge, no alternative measure was suggested prior to Theorem \ref{thm:thm} above. 
We note that forming $Z$ is expensive computationally and requires the inversion of $A$ and $H_p$, hence it should not be seen as a predictive tool, rather as an analytical one, which explains the convergence patterns we see in practice.
Furthermore, although Theorem \ref{thm:thm} is formulated for the specific case of our block-acoustic preconditioner, it gives rise to similar results for other approximate commutator preconditioners, which we leave for future investigation.

\section{Numerical results}\label{sec:results}

\begin{figure}
\begin{center}
	\newcommand{\image}[1]{\includegraphics[width=0.32\linewidth]{#1}}
    \subfigure[\footnotesize $\rho$]{\image{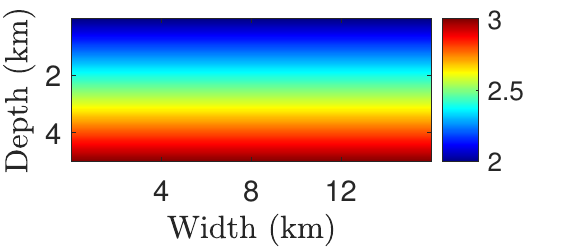}} \label{fig:linear_rho}
    \subfigure[\footnotesize $V_p$]{\image{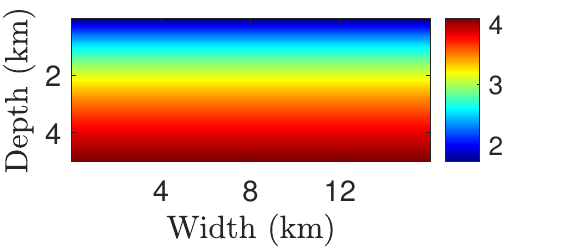}}\label{fig:linear_vp}    
        \subfigure[\footnotesize $V_s$]{\image{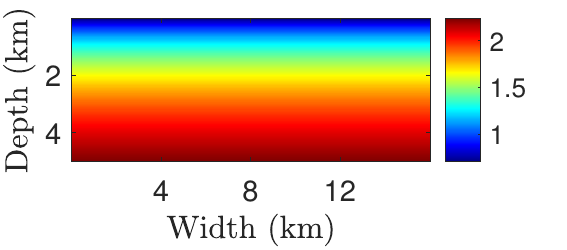}}\label{fig:linear_vs}    
\\
\end{center}
\caption{The elastic linear model. Velocity units: $km/sec$, density units: $kg/m^3$.
}\label{fig:Linear}
\end{figure}

\begin{figure}
\begin{center}
	\newcommand{\image}[1]{\includegraphics[width=0.32\linewidth]{#1}}
    \subfigure[\footnotesize $\rho$]{\image{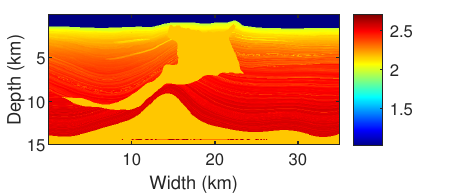}} \label{fig:SEAMrho}
    \subfigure[\footnotesize $V_p$]{\image{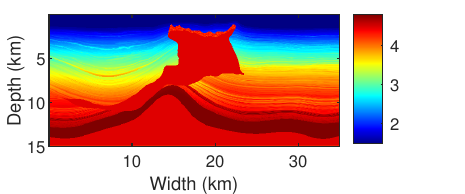}}\label{fig:SEAMvp}    
        \subfigure[\footnotesize $V_s$]{\image{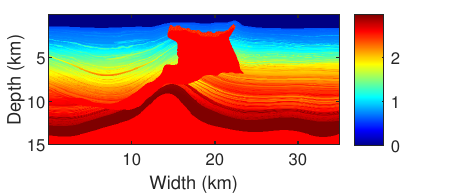}}\label{fig:SEAMvs}    
\\
\end{center}
\caption{Elastic version of SEG advenced modeling (SEAM phase I). Velocity units: $km/sec$, density units: $kg/m^3$.
}\label{fig:SEAM}
\end{figure}

In this section we perform numerical experiments to establish the efficiency and robustness of our method.
In Subsection \ref{subsec:scalability} we demonstrate scalability properties of our preconditioner, and in Subsection \ref{subsec:thm_results} we verify the theoretical result of Theorem \ref{thm:thm} and its implications. 
In Subsection \ref{subsec:approx_commutator_comparison} we compare our block-acoustic preconditioner to other approximate commutator preconditioners, and discuss the limitations of this comparison. 
In Subsection \ref{subsec:monolithic_comparison},
we show the superiority of our preconditioner --- with a shifted Laplacian multigrid solve of each block --- over the monolithic multigrid method suggested in \cite{treister2024hybrid}.
Finally, in Subsection \ref{subsec:3d} we demonstrate the effectiveness of our preconditioner in solving real-world 3D problems with challenging geophysical media.

Throughout the experiments, we solve the dicretized elastic Helmholtz problem in mixed formulation \eqref{eq:mixed-discretized} on a finite rectangular domain (or cubic, in 3D), subject to a point-source located in the center of the top edge (or face) of the domain, up to a tolerance of relative residual $<10^{-6}$. 
The ABC are implemented as a layer of gradually increasing attenuation $\gamma$, of width of 20 grid cells.
This is additional to a basic natural attenuation of $\gamma=0.01\pi$ that we assume throughout the domain, in all of our experiments.
The default frequency corresponds to at least $10$ grid points per shear wavelength. 
That is, in heterogeneous media and unless specified otherwise, the most challenging region has about $10$ grid points per shear wavelength.

Our code is written in the {\tt Julia} language \cite{Julia}, and is included as a part of the {\tt jInv.jl} package \cite{ruthotto2017jinv}. This package enables using of our code as a forward solver for elastic full waveform inversion in the frequency domain. Our tests were computed on a simple dual-core laptop with 32 GB RAM, running Windows 10.

For the 2D experiments in the next three subsections, we use the following models:
\begin{itemize}
\item
\textbf{Homogeneous media:} a dimensionless domain $\Omega = [0,16]\times[0,5]$, with constant $\rho,$ $\lambda$ and $\mu$. Our default values are $\rho=\mu=1$ and $\lambda=16$, which corresponds to Poisson ratio $\sigma = 0.47$. 
As noted before, in our discretization, the commutation is exact for this media in the interior of the domain.
However, due to the application of ABC, the coefficient $\gamma$ non-constant, even for constant media. 
\item
\textbf{Linear media:} a domain of width 16 $km$ and depth 5 $km$, with density and Lam{\'e} coefficients that varies linearly in the vertical dimension. Our default ranges are $\rho\in[2,3],$ $\mu\in[1,15]$ and $\lambda\in[4,20]$, which corresponds to Poisson ratio of at most $\sigma = 0.4$. Figure \ref{fig:Linear} shows the density and wave velocities. 
\item
\textbf{SEAM phase I model:} (SEG Advanced Modeling). A geophysical elastic model of 35 $km$ width and 15 $km$ depth \cite{fehler2008seg} with density and wave velocities as in Figure \ref{fig:SEAM}.
We exclude the seabed on the top of the domain.
\item
\textbf{SEG/EAGE salt model:} a geophysical acoustic model of 4.2 $km$ depth and 13.5 $km$ width, see \cite{ricker}, with pressure wave velocity as in Figure \ref{fig:SEG}. 
Inspired by \cite{brossier2010data}, we extend it to elastic model by assuming a constant Poisson ratio $\sigma=1/3$, reflected by the choice $V_s=0.5V_p$, 
and we define
$\rho = 0.25V_p+1.5$, to yield values that ranges similarly to the elastic Marmousi2 model below.
\item
\textbf{Marmousi2 media:} an elastic 2D geophysical model suggested in \cite{martin2006marmousi2}, based on a section of the Kwanza basin in Angola. The domain is very shallow: 17 $km$ width and only 3 $km$ depth, so we add a vertical extension of 16 cells in the bottom, to facilitate the application of the ABC. Figure \ref{fig:Marmousi} shows the (non-extended) model.
\end{itemize}

\begin{figure}
  \centering
  \includegraphics[width=0.32\textwidth]{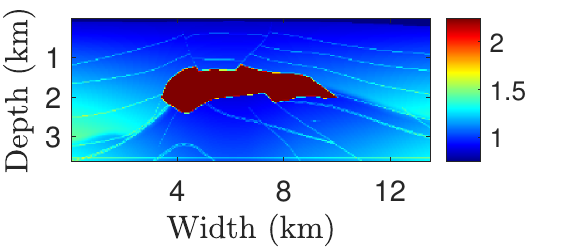}\\
  \caption{SEG/EAGE salt model. Velocity units: $km/sec$}\label{fig:SEG}
\end{figure}

\begin{figure}
\begin{center}
	\newcommand{\image}[1]{\includegraphics[width=0.32\linewidth]{#1}}
    \subfigure[\footnotesize $\rho$]{\image{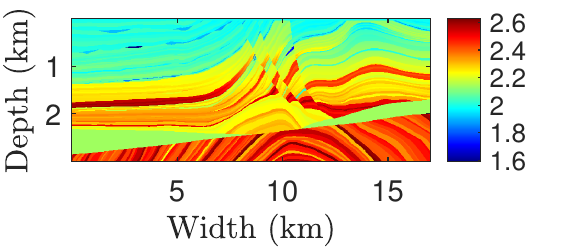}} \label{fig:MarmousiRho}
    \subfigure[\footnotesize $V_p$]{\image{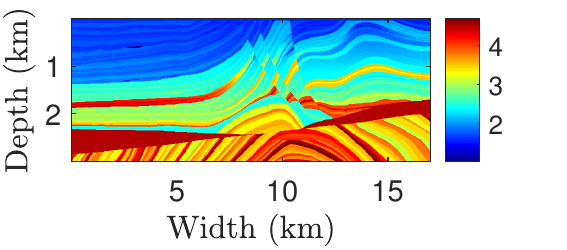}}\label{fig:MarmousiVp}    
        \subfigure[\footnotesize $V_s$]{\image{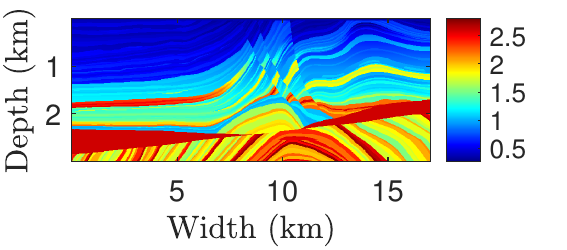}}\label{fig:MarmousiVs}    
\\
\end{center}
\caption{The elastic Marmousi2 2D model. Velocity units: $km/sec$, density units: $kg/m^3$.
}\label{fig:Marmousi}
\end{figure}

\begin{table} 
\centering
\begin{tabular}{c|ccccc}
\toprule
  \mc{6}{c}{Average magnitude of $\Xi$}\\
\midrule
& \small Const & \small Linear & \small SEAM & \small SEG salt & \small Marmousi \\
\midrule
\small $G_s = 10,\alpha = 0.0$ & 0.00383 & 0.13041 & 0.24057 & 0.63255 & 2.60292 \\
\small $G_s = 10,\alpha = 0.1$ & 0.00383 & 0.13041 & 0.24057 & 0.63254 & 2.60292 \\
\small $G_s = 100,\alpha=0.0$ & 3.74e-5 &  0.12925 & 0.23989 & 0.62605 & 2.60024 \\
\bottomrule
 \end{tabular}
\caption{Average magnitude of an element in the commutator $\Xi$ for different models. $G_s$ is the number of grid points per shear wavelength and $\alpha$ is the shift parameter.}
\label{tab:commutator_models}
\end{table}

In Table \ref{tab:commutator_models} we measure the average magnitude of a commutator entry for different scenarios.
The results demonstrate that the Marmousi2 model presents a notably higher level of difficulty compared to other models.
It is also shown that for highly heterogeneous media, changes in frequency (reflected by $G_s$, the number of grid points per shear wavelength) or complex shift $\alpha$ have negligible effect on the commutator.

\subsection{Scalability}\label{subsec:scalability}

In the first experiment, we show the scalability properties of our preconditioner.
Table \ref{tab:scalability_linear} counts the non-restarted GMRES iterations needed to solve the linear model problem for different grid sizes.
We repeat the experiment for three more variants of the default linear model, in which the $\lambda$ range is multiplied by $10$, $100$ and $1000$.
The aim of this modification is enlarging the Poisson ratio while retaining the smoothness of the media.
The iteration count does not exceed $19$, for any of the grid sizes and Poisson ratios examined.
That is, the method is scalable with respect to the Poisson ratio and with respect to the grid size.

\begin{table} 
\centering
\begin{tabular}{cl|cccc}
\toprule
  \mc{6}{c}{Non-restarted GMRES iteration count, linear media}\\
\midrule
Grid size (cells) & frequency & $\,\,\,\lambda\cdot1\,\,\,$ & $\,\,\lambda\cdot10\,\,$ &  $\,\lambda\cdot100\,$ &  $\lambda\cdot1000$ \\
\midrule
$200\times 64$ & $\omega = 1.75\pi$ & 14 & 13 & 13 & 13 \\
$400\times 128$ & $\omega = 3.5\pi$ & 14 & 13 & 12 & 12 \\
$800\times 256$ & $\omega = 7\pi$ & 14 & 16 & 12 & 11 \\
$1600\times 512$ & $\omega = 14\pi$ & 14 & 19 & 12 & 12 \\
\bottomrule
 \end{tabular}
\caption{Number of block-acoustic preconditioning cycles needed for convergence for the 2D elastic Helmholtz equation in linear media. 
Each acoustic block is solved directly.
Linear media with different Poisson ratios  were implemented by a point-wise multiplication of $\lambda$ by different factors.  
The highest factor applied to $\lambda$ corresponds to a Poisson's ratio $\sigma\geq0.4996$ all over the grid.}
\label{tab:scalability_linear}
\end{table}

The latter is analogous to a property sought in many works regarding the acoustic Helmholtz equation: \emph{wavenumber independent convergence.}
Table \ref{tab:scalability_linear} shows a nearly constant iteration count, regardless of the frequency, which is chosen for each grid size to keep a ratio of about $10$ grid points per shear wavelength in the most challenging regime of the domain.
However, in large real-world instances, it is impractical to solve each block directly as we do here, and most existing iterative solvers for acoustic Helmholtz require a shift --- which interferes with the scalability with respect to the grid size.
Yet, the results above gives rise to future investigation of the acoustic Helmholtz equation: 
once an acoustic scalable solver will be developed, it will be automatically applicable for the elastic case via our preconditioner, without harming the scaling properties.

\subsection{Demonstration of Theorem \ref{thm:thm}}\label{subsec:thm_results}

In this subsection we verify numerically the result from Theorem \ref{thm:thm} and study its implications.
Recall that $Z$, defined in \eqref{eq:Z}, is an $n_{cells}\times n_{cells}$ matrix that --- as proved in Theorem \ref{thm:thm} --- captures all the spectral information on the error iteration matrix $T$.

In Figure \ref{fig:spec} we depict the spectrum of $Z$ for different media and different frequencies, represented as different numbers of grid points per shear wavelength, $G_s$. 
Since forming $Z$ and calculating its entire set of eigenvalues is very costly, we hold this experiment for a small $50\times 50$ cells slice of the homogeneous media, of the linear media (of original size $200\times 64$ cells) and of Marmousi2 media (of original size $544\times 128$ cells). 
We stress that this slicing causes loss of some of the heterogeneity encapsulated by the media, therefore, we draw only qualitative conclusions from the comparison.
We observe that the spectrum becomes more scattered as the heterogeneity and non-smoothness grows. 
Moreover, we observe that when a higher frequency is taken for the same media, some eigenvalues scatter far away from zero. 
As expected from the discussion in Subsection \ref{subsec:thm}, the near zero eigenvalues of the highly indefinite $A$ causes this phenomenon, that can be relaxed by an additional shift.

\begin{figure}
\begin{center}
	\newcommand{\image}[1]{\includegraphics[width=0.32\linewidth]{#1}}
    \subfigure[\footnotesize Homogeneous, $G_s=100$]{\image{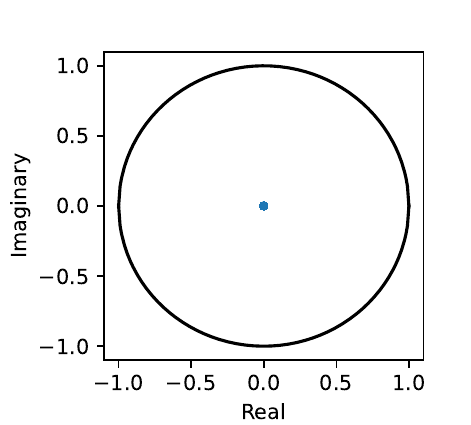}\label{fig:eigsZ_const_50_50}} 
    \subfigure[\footnotesize Linear, $G_s=100$]{\image{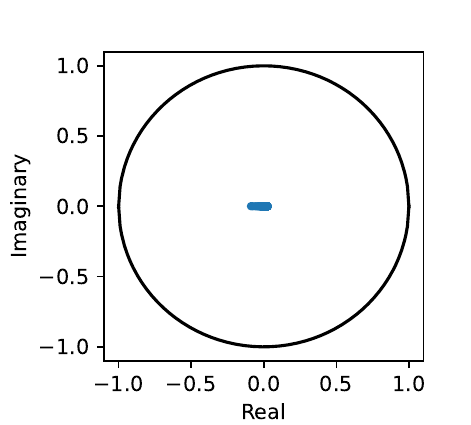}\label{fig:eigsZ_linear_50_50}} 
    \subfigure[\footnotesize Marmousi2, $G_s=100$]{\image{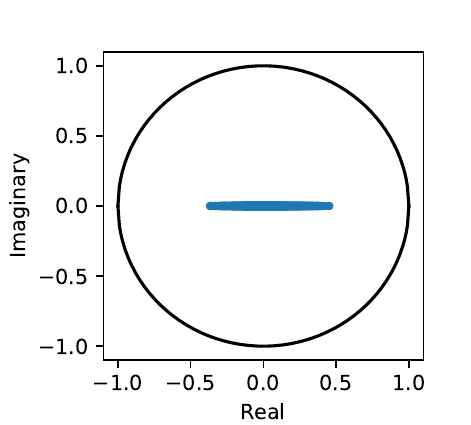}\label{fig:eigsZ_Marmousi_50_50}}    
\\
\subfigure[\footnotesize Homogeneous, $G_s=10$]{\image{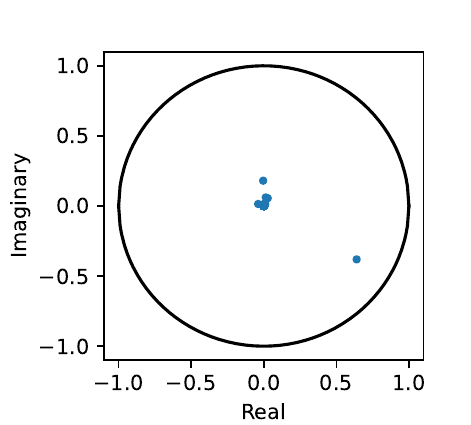}\label{fig:eigsZ_const_50_50_omega1}} 
    \subfigure[\footnotesize Linear, $G_s=10$]{\image{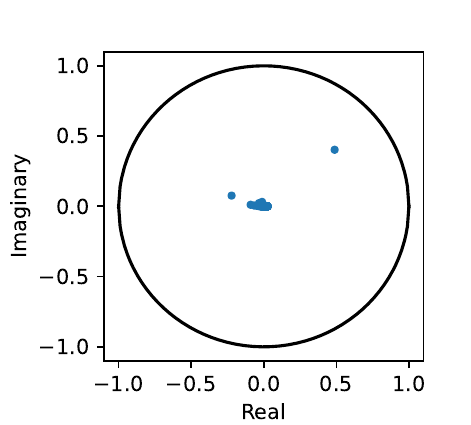}\label{fig:eigsZ_linear_50_50_omega1}} 
    \subfigure[\footnotesize Marmousi, $G_s=10$]{\image{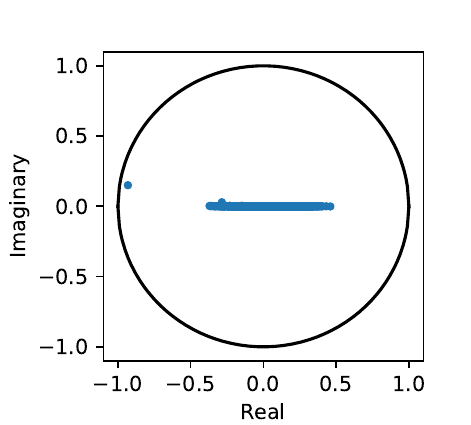}\label{fig:eigsZ_Marmousi_50_50_omega1}}
    \\
\end{center}
\caption{The spectrum of $Z$ from \eqref{eq:Z} for a $50\times 50$ cells slice of different media and frequencies.
}\label{fig:spec}
\end{figure}

Figure \ref{fig:rhoZ} shows the spectral radius of $Z$ for different models as a function of the shift. 
We implemented this by applying the power method on a mat-vec code of $Z$, without forming the matrix. 
We observe that increasing the shift lowers the spectral radius, as expected.
For linear, SEG or SEAM media of the given size, the minimal shift needed for the spectral radius to be lower than 1, is zero or negligible, as it is in the same order of magnitude as the natural attenuation that we use in our experiments. 
For Marmousi2 media of the given size, a larger shift of about $0.1$ is required for the spectral radius to be lower than 1.
A similar experiment for extremely heterogeneous (yet smooth) media is shown in
Figure \ref{fig:rhoZ_linear_extreme}, demonstrating the larger shift needed in this case. 
This synthetic model (which is non-realistic for most geophysical applications) is defined by stretching the range of $\rho$, $\lambda$ and $\mu$ of the original linear media by up to $10^6$, resulting in $\mu$ variations of up to 7 orders of magnitude. 
The range of $\mu$ values is our reference point, since it has the largest effect on the commutator.

\begin{figure}
\begin{center}
	\newcommand{\image}[1]{\includegraphics[width=0.45\linewidth]{#1}}
    \subfigure[Geophysical media]{\image{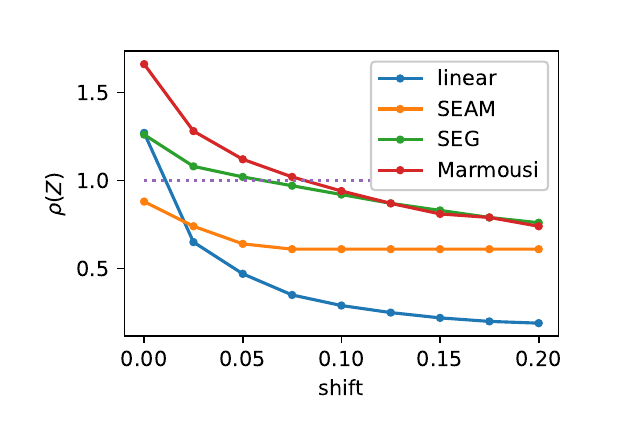}\label{fig:rhoZ}}    
    \subfigure[Extreme heterogeneity]{\image{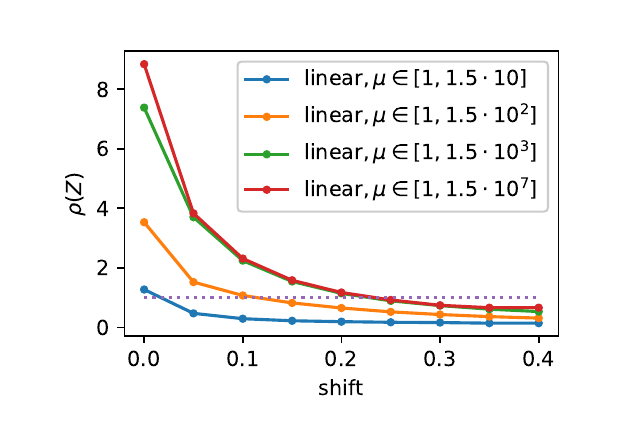}\label{fig:rhoZ_linear_extreme}}   
\\
\end{center}
\caption{The spectral radius of $Z$ from \eqref{eq:Z} as a function of the shift.
On the left, for linear media of size $400\times 128$ cells, SEG/EAGE salt media of size $160\times 320$, SEAM phase I media of size $256\times 512$ and Marmousi2 media of size $544\times 112$ cells (before the vertical extension).
On the right, for similar stretched linear media with $\mu$ variations of up to 7 orders of magnitude.
}\label{fig:rhoZ_both}
\end{figure}

Since computing spectral properties of $Z$ is not practical in large cases, we next examine how $Z$ acts on a vector.
We suggest that this behavior is governed by the interaction between the commutator and the acoustic Helmholtz inverse operators $A^{-1}$ and $H_p^{-1}$.
These operators may amplify or damp error differently, and they may also affect its smoothness.
We suggest that the commutator does not only vary in norm across different media, but also changes its action at different frequencies, depending on the smoothness of the vectors on which it is applied.

We investigate the action of $Z$ on the leading eigenvector.
We see, surprisingly, that most of the difference between the high-frequency and low-frequency cases comes from the commutator's action, although the two commutators are equal in norm.
Let $\lambda$ be the largest eigenvalue in Figure \ref{fig:eigsZ_linear_50_50_omega1} and $\bfv$ its corresponding normalized eigenvector.
Let $A,B,H_p,\Xi$ and $Z$ be the corresponding operators for the high-frequency problem of Figure \ref{fig:eigsZ_linear_50_50_omega1}, and 
$\tilde{A},\tilde{B},\tilde{H}_p,\tilde{\Xi}$ and $\tilde{Z}$ be their counterparts for the low-frequency problem of Figure \ref{fig:eigsZ_linear_50_50}.
Then, the total amplification differs by an order of magnitude: $\|Z\bfv\|_2 = 0.69$ versus $\|\tilde{Z}\bfv\|_2 = 0.078$.
Yet, most of this difference comes from the commutator $\Xi$, and not from the inverses $A^{-1}$ and $H_p^{-1}$,
\begin{equation} \label{eq:route_stray_eig}
\frac{\|H_p^{-1}\bfv\|_2}{\|\tilde{H}_p^{-1}\bfv\|_2} \approx 1.2, \qquad \frac{\|A^{-1}B^T H_p^{-1}\bfv\|_2}{\|\tilde{A}^{-1}\tilde{B}^T \tilde{H}_p^{-1}\bfv\|_2} \approx 2.5,
\end{equation}
which only magnify by a factor of about $1.2-2$ each.

Our suggested explanation to this phenomenon is twofold.
First, the commutator seems to act differently when applied to smooth versus rough vectors.
Second, the wavy nature of solutions for $A$ and $H_p$ (that is, applications of their inverses to given vectors) may affect the smoothness or roughness of the vectors.
Taken together, the commutator applied to these vectors can produce different outcomes --- even when using essentially the same commutator.

To demonstrate the influence of $A$ on the smoothness or roughness of its solutions, we perform another experiment. 
Let $\Xi$ be the commutator corresponding to the setting of Figure \ref{fig:eigsZ_linear_50_50_omega1} (for the full $200\times64$ cells media without slicing). 
Let $A_{high}$ be its corresponding $A$ block, $A_{low}$ the corresponding $A$ block for the low frequency problem of Figure \ref{fig:eigsZ_linear_50_50}, and let $A_{shifted}$ be similar to $A_{high}$ with an added complex shift $\alpha = 0.5$.
Let $\bfv$ be a random vector whose dimensions enables the product $\Xi \bfv$. 
We first produce algebraically-smooth vectors: 
we approximate $A_{high}^{-1}\bfv$ by applying 10 GMRES(10) iterations, and take $\bfe_{high}$ to be the error. 
This algebraically smooth vector with respect to $A_{high}$ is not geometrically smooth, as seen in Fig. \ref{fig:alg_smooth_A_high_linear}, since $A_{high}$ prioritize wavy solutions.
Similarly, we produce $\bfe_{low}$ and $\bfe_{shifted}$, which both have entries smaller in magnitude --- by factors of 9--10 and 3--4, respectively --- compared to $\bfe_{high}$,  notice the color-bars in Fig. \ref{fig:alg_smooth_A_low_linear} and \ref{fig:alg_smooth_A_shifted_linear}. 
For the Poisson-like $A_{low}$, the error is geometrically smooth.
Due to space constraints, we do not include here similar figures for SEG, SEAM and Marmousi2 media. 
From our experience, in highly nonsmooth media, all the algebraically-smooth vectors --- even w.r.t. $A_{low}$ --- are not geometrically smooth.

\begin{figure}
\begin{center}
	\newcommand{\image}[1]{\includegraphics[width=0.32\linewidth]{#1}}
    \subfigure[\footnotesize smooth w.r.t. $A_{high}$]{\image{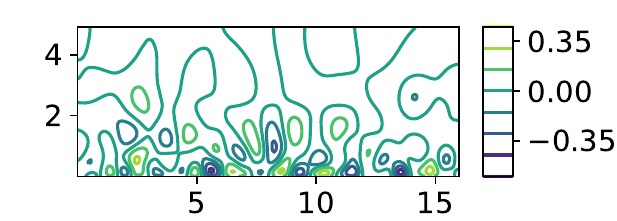}\label{fig:alg_smooth_A_high_linear}} 
    \subfigure[\footnotesize smooth w.r.t. $A_{low}$]{\image{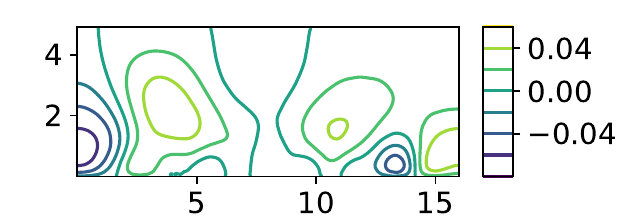}\label{fig:alg_smooth_A_low_linear}}   
        \subfigure[\footnotesize smooth w.r.t. $A_{shifted}$]{\image{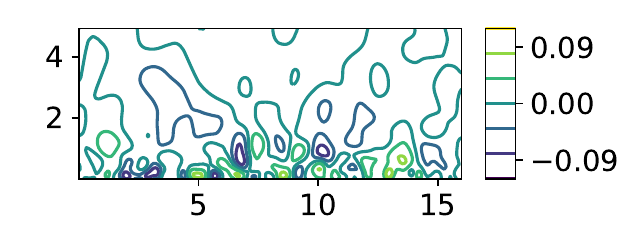}\label{fig:alg_smooth_A_shifted_linear}}    
\\
\end{center}
\caption{The vectors $\bfe_{high},\bfe_{low}$ and $\bfe_{shifted}$, obtained by approximating, e.g., $A_{high}^{-1}\bfv$ using 10 GMRES(10) iterations, for a random $\bfv$.
}\label{fig:alg_smooth}
\end{figure}

To demonstrate that the commutator might act differently on smooth and rough vectors, we apply $\Xi$ to a normalized version of the vectors produced above. We report the resulting amplification in Table \ref{tab:Xi}.
The $h^3$ factor compensates for $1/h^3$ that appears in $\Xi$.
For linear media, we see a significant damping of the error by the commutator in the low frequency case.
Together with the smoothness of $\bfe_{low}$ shown in Fig. \ref{fig:alg_smooth_A_low_linear}, it may suggest that the commutator damps smooth vectors more than it damps rough vectors.
The same behavior is observed for the (less-smooth) SEG model, although the effect is less significant.
For highly nonsmooth media, even the low-frequency algebraically smooth vectors are not necessarily geometrically smooth, and as a result, all the numbers in Table \ref{tab:Xi} for Marmousi media are similar. 
For the SEAM model, where the pattern of jumps in coefficients is similar to Marmousi2 yet the magnitude of the jumps is smaller, the qualitative behavior is the same, though the damping of errors is generally better.
The similarity of the results for \emph{normalized} versions of $\bfe_{high}$ and $\bfe_{shifted}$ can be misleading: note that
before normalization the magnitude of $\bfe_{shifted}$ is smaller, as shown in Fig. \ref{fig:alg_smooth_A_shifted_linear}.  We suggest that the effect of the shift in improving convergence is not related to smoothing, but to lowering the magnitude of the result.
Finally, one could infer that lower frequency and higher attenuation  improve the convergence, although this is not directly seen by Eq. \eqref{eq:route_stray_eig} following the route of a single eigenvector.

\begin{table} 
\centering
\begin{tabular}{l|cccc}
\toprule
  \mc{5}{c}{$h^3 \| \Xi \bfe\| / \|\bfe\|$}\\
\midrule
 & \small linear & \small SEAM  & \small SEG & \small Marmousi \\
\midrule
$\bfe = \bfe_{high}$ & 0.02177 &  0.12539 & 0.06387 & 0.51731 \\
$\bfe = \bfe_{low}$ & 0.00899 & 0.12202 & 0.04288 & 0.50624 \\
$\bfe = \bfe_{shifted}$ & 0.02222 & 0.12734 & 0.06308 & 0.51706 \\
\bottomrule
 \end{tabular}
\caption{Application of $\Xi$ to algebraically smooth vectors.}
\label{tab:Xi}
\end{table}

\subsection{Comparison with other approximate commutator methods}\label{subsec:approx_commutator_comparison}

In this subsection we compare the block-acoustic preconditioner to existing preconditioners described in Subsection \ref{subsec:approx_com_background}.
Nevertheless, the latter were developed for incompressible fluid flow problems, and are described in terms of finite elements discretization. 
We apply a naive implementation of $F_p$ from \eqref{eq:fp}\footnote{We took the left equation of \eqref{eq:fp}, but observed that the other version behaves similarly in all of our experiments.} and BFBt from \eqref{eq:wbfbt}, replacing the finite elements mass matrices by identity matrices, and the convection-diffusion leading block by our $A$ from \eqref{eq:mixed-discretized}.

Needless to say, the comparison of methods that were developed for such different problems is not quite equitable. 
There are two main differences between the Stokes-like systems and elastic Helmholtz saddle-point system:
the indefiniteness of the leading block, and the existence of a non-zero $C$ block.
We suggest two adaptations, to accommodate for the differences: testing the performance for a wide range of frequencies, including $\omega\approx0$ for which the leading block is SPD, and testing different Poisson ratios, including a nearly incompressible case.

In the first experiment, we apply the block-acoustic, $F_p$ and BFBt preconditioners to a model problem of homogeneous media with a grid size of $128\times 64$ cells, for various values of $\omega$. 
Figure \ref{fig:approx_com_const_lambda1} shows the GMRES(5) iteration count as a function of the frequency.
The smallest $\omega$ that we take is negligible (corresponds to $10^4$ grid points per shear wavelength), and resembles a linear elasticity problem.
The largest $\omega$ corresponds to about $11$ grid points per wavelength. 
For the linear elasticity problem, the preconditioners function similarly, whereas for high-frequency elastic Helmholtz, the block-acoustic preconditioner outperforms the other examined approximate commutator preconditioners. 
In the second experiment, we repeat the first experiment for nearly incompressible media. That is, an identical shear wave velocity and a higher Poisson ratio. 
In Figure \ref{fig:approx_com_const_lambda1000} we see that the three preconditioners solve the problem by a similar rate in the nearly incompressible case.
This result is somewhat surprising and requires further investigation, as the $F_p$ and $BFBt$ preconditioners were \emph{not} originally designed for saddle-point systems with an indefinite leading block, and to the best of our knowledge, were not investigated in this context previously.

To sum up, in linear elasticity problems, low frequency problems (of about 20 grid points per shear wavelength or more) and in the nearly incompressible case --- the methods preform similarly. 
Though, the aim of the elastic Helmholtz equation is modeling waves in \emph{solids}, which are compressible. 
Incompressible elastic Helmholtz is in fact reducible to acoustic Helmholtz, as shown in Subsection \ref{subsec:helm_background}. 
Moreover, we performed the above mentioned comparison on homogeneous media only, since the methods from Subsection \ref{subsec:approx_com_background}, in our implementation, preformed poorly for linear and Marmousi2 media.

\begin{figure}
\begin{center}
	\newcommand{\image}[1]{\includegraphics[width=0.4\linewidth]{#1}}
    \subfigure[\footnotesize $\sigma = 0.470$ (compressible)]{\image{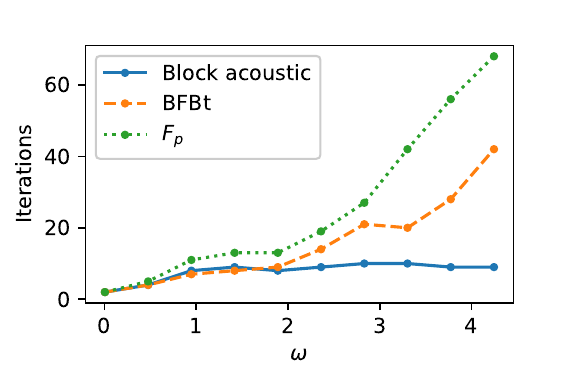}} \label{fig:approx_com_const_lambda1}
    \hspace{10pt}
    \subfigure[\footnotesize $\sigma = 0.499$ (neary incompressible)]{\image{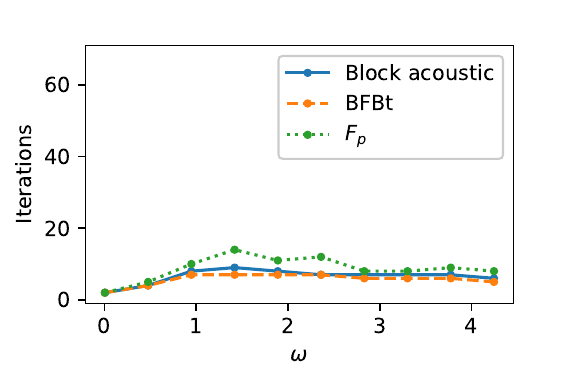}}\label{fig:approx_com_const_lambda1000} 
\\
\end{center}
\caption{Number of preconditioned GMRES(5) iterations needed for convergence for the elastic Helmholtz equation in constant media, as a function of the frequency.
The lowest frequency is negligible and corresponds to about $10^4$ grid points per shear wavelength, and the highest frequency corresponds to about 11 grid points per shear wavelength.
On the left, the default constant media was taken, and on the right, it was altered by multiplying the lam{\'e} coefficient $\lambda$ by $1000$.}
\label{fig:approx_com_const}
\end{figure}

\subsection{Monolithic vs. block-acoustic multirgid}\label{subsec:monolithic_comparison}

In this section we compare the performance of the block-acoustic multigrid preconditioner --- our preconditioner combined with CSLP solve of each acoustic block --- to the monolithic  multigrid preconditioner from \cite{treister2024hybrid}, described in Subsection \ref{subsec:sl}.
We first present the multigrid setup of the two methods, followed by a FLOP calculation for each. Finally, we show the iteration count and total computational cost for different media.

\subsubsection{Multigrid setup}

First, we describe the multigrid setup for the block-acoustic multigrid preconditioning.
In each preconditioning step, we approximately solve each block of the block-acoustic preconditioner by a $W(1,2)$ cycle of CSLP \cite{erlangga2006novel}, with damped Jacobi as a smoother. The damping parameters on the first, second and third levels are $0.8,0.8$ and $0.3$. 

We use different intergrid operators for each diagonal block, corresponding to the locations on the staggered grid. 
The diagonal blocks in \eqref{eq:prec}
(in 2D) are $A_1$, discretized on the $\bfu_1$ faces, $A_2$ on the $\bfu_2$ faces, and $H_p$  which is cell centered. For each of them, we take the restriction to be a Kroncker product of the 1D restrictions defined by the stencil $\begin{bmatrix}1 & 2 & 1 \end{bmatrix}$ for the nodal direction and by the stencil $\begin{bmatrix}1 & 3 & * & 3 & 1 \end{bmatrix}$ for the cell-centered direction. For instance, $\bfu_1$ is nodal in the $x$ direction and cell-centered in the $y$ direction and hence 
\begin{equation}\label{eq:Ru1}
R_{u_1} = \frac{1}{4}\begin{bmatrix}1 & 2 & 1 \end{bmatrix} \,\otimes\, \frac{1}{8}\begin{bmatrix}1 & 3 & * & 3 & 1 \end{bmatrix}
\end{equation}
is the restriction for the block $A_1$. The restriction $R_{u_2}$ is defined similarly. The restriction for $H_p$ is
\begin{equation}\label{eq:Rp}
R_p = \frac{1}{8}\begin{bmatrix}1 & 3 & * & 3 & 1 \end{bmatrix} \,\otimes\, \frac{1}{8}\begin{bmatrix}1 & 3 & * & 3 & 1 \end{bmatrix}.
\end{equation}
The prolongations are defined as $P_{u_1}=2R_{u_1}^T$ and similarly for $P_{u_2}$ and $P_p$. 

The coarse grid operator is determined by Galerkin coarsening. 
That is, e.g.,
\begin{equation}\label{eq:galerkin}
(H_p)_H = R_p(H_p)_hP_p.
\end{equation}
The coarsest grid problem is solved directly using LU decomposition.

For the monolithic multirgid, we use Vanka red-black smoother with relaxation parameters of $0.65,0.5$ and $0.3$ for the first, second and third grid, respectively. We use $W(1,1)$ cycles with Galerkin coarsening and solve the coarsest grid directly. As intergrid operators, we take a mixed version: the prolongation is
\begin{equation}\label{eq:P}
P = \text{blockdiag}(P_{u_1},P_{u_2},P_p)
\end{equation}
where $P_{u_1},P_{u_2}$ and $P_{p}$ are as in the block-acoustic case. However, the restriction is \emph{not} a scaled transpose of the prolongation: 
we take $R=\text{blockdiag}(\tilde{R}_{u_1},\tilde{R}_{u_2},\tilde{R}_p)$ where  
\begin{equation}\label{eq:tildeRu1}
\tilde{R}_{u_1} = \frac{1}{4}\begin{bmatrix}1 & 2 & 1 \end{bmatrix} \, \otimes \, \frac{1}{2}\begin{bmatrix}1 & * & 1 \end{bmatrix}
\text{ and }
\tilde{R}_p = \frac{1}{2}\begin{bmatrix}1 & * & 1 \end{bmatrix} \, \otimes \, \frac{1}{2}\begin{bmatrix}1 & * & 1 \end{bmatrix}
\end{equation}
and $\tilde{R}_{u_2}$ is defined similarly to $\tilde{R}_{u_1}$.

The choice of mixed intergrid in the monolithic case reduces the complexity of the operators, while keeping almost the same convergence rate, as mentioned in \cite{treister2024hybrid,yovel2024lfa}. 
For the block-acoustic case in 2D, however, we have seen by trial and error that the mixed intergrid is less favorable.
The different number of relaxation cycles in the $W(1,2)$ Jacobi compared to $W(1,1)$ Vanka, compensates for the additional computational cost. 
In fact, Vanka relaxation has about twice the computational cost of Jacobi relaxation (as calculated in the next subsection), but taking more than 3 relaxations per cycle did not lead to any additional improvement in convergence of the block-acoustic multigrid. 
For a more thorough explanation on the choice of multigrid components, see Subsection \ref{subsubsec:sensitivity}.

\subsubsection{FLOP count}

Since the comparison of the two different multigrid frameworks is not equitable, we first include a detailed FLOP count for each of the methods.
The results are displayed in Table \ref{tab:nnz}, including the sum of two FLOP sources: the analytically estimated computational cost for the entire preconditioning step except for the coarse solve, and the number of nonzeros in the resulting LU decomposition. 
Below, we demonstrate the analytic estimation, not including the coarsest solve, for the case of a two-level cycle.

\begin{table} 
\centering
\begin{tabular}{c|ccc|ccc}
\toprule
  \mc{7}{c}{FLOP count per cell}\\
\midrule
 &  \mc{3}{c|}{Block-acoustic multigrid} & \mc{3}{c}{Monolithic multigrid} \\
\midrule
Grid size& \small 2-level & \small 3-level & \small 4-level & \small 2-level  & \small 3-level & \small 4-level \\
\midrule
$400\times 128$ & 193 & 261 & 332 & 349 & 355 & 486 \\
$800\times 256$ & 223 & 274 & 340 & 446 & 396 & 501 \\
$1600\times 512$ & 277 & 291 & 346 & 558 & 445 & 522 \\
%\midrule
%$544\times 112$ & 195 & 262 & 332 & 346 & 354 & 487 \\
%$1088\times 224$ & 222 & 276 & 339 & 434 & 390 & 501 \\
%$2176\times 448$ & 272 & 290 & 347 & 519 & 432 & 519 \\
%\midrule
%$64\times 64$ & 151 & 239 & 318 & 244 & 318 & 467 \\
%$128\times 128$ & 181 & 253 & 327 & 318 & 343 & 482 \\
%$256\times 256$ & 206 & 268 & 335 & 415 & 380 & 495 \\
%\midrule
%$128\times 64$ & 162 & 245 & 322 & 273 & 329 & 472 \\
%$256\times 128$ & 193 & 258 & 330 & 368 & 357 & 488 \\
%$512\times 256$ & 233 & 274 & 338 & 477 & 405 & 502 \\
\bottomrule
 \end{tabular}
\caption{The total FLOP count per cycle per discretization cell, including the number of nonzeros of the LU factors. 
In the block acoustic case, for three $W(1,2)$ cycles with Jacobi relaxation for the blocks $A_1$, $A_2$ and $H_p$. In the monolithic case, for a $W(1,1)$ cycle with Vanka relaxation.}
\label{tab:nnz}
\end{table}

We count the costs of residual, relaxation and intergrid calculation for each method.
For the block-acoustic method, we add the cost of forming the corrected right-hand-side in \eqref{eq:mixed-distributed}.
We estimate the FLOPs by counting number of nonzeros.
Denoting the number of cells by $n_{cells}$, we neglect terms of order $\sqrt{n_{cells}}$.  

In each block-acoustic multigrid  preconditioning step, in 2D, three $W(1,2)$ cycles are applied, one for each of the 5-diagonal matrices $A_1$, $A_2$ and $H_p$ of size $n_{cells} \times n_{cells}$.
The residual calculation costs $15n_{cells}$. 
A Jacobi step includes residual calculation and an additional $1n_{cells}$, resulting in $54n_{cells}$.
The restriction $R_{u_1}$ from \eqref{eq:Ru1} and $R_{u_2}$ has 12 nonzeros per row, and the restriction $R_p$ from \eqref{eq:Rp} has 16 nonzeros per row, leading to a total cost of $20n_{cells}$ for intergrid operations.
The calculation of the corrected right-hand-side in \eqref{eq:mixed-distributed} costs $9n_{cells}$, because $A_p$ is 5-diagonal and $B$ is a concatenation of two 2-diagonal matrices. 
To sum up, excluding the LU solve in the coarse grid, one preconditioning cycle costs $98n_{cells}$ FLOPs.

For the monolithic multigrid, the residual computation sums up to $19n_{cells}$ nonzeros (counting the diagonals of all blocks in \eqref{eq:mixed-discretized}). 
One Vanka relaxation includes a residual calculation and additional $17n_{cells}$ FLOPs, when the special structure of the $5\times 5$ submatrices is exploited, as noted in \cite{yovel2024lfa}, Remark 3.1. 
The total relaxation cost for a $W(1,1)$ cycle is thus $72n_{cells}$. 
The prolongation cost is identical to the block-acoustic case, and the restriction is slightly cheaper: $\tilde{R}_{u_1}$ and $\tilde{R}_p$ from \eqref{eq:tildeRu1} has 6 and 4 nonzeros per row, respectively, each having $0.25n_{cells}$ rows. Hence, the total intergrid cost is $14n_{cells}$. It sums up to $105n_{cells}$ per cycle, excluding the LU cost.

A major part of the total computational cost is invested in the coarse grid solve.
In 2D, the coarse grid LU factors of the monolithic multigrid are $2.5$ times heavier than the three pairs of factors in the block-acoustic multigrid altogether.
It can be explained by the coupling of the multi-diagonal blocks in the elastic Helmholtz operator in mixed formulation \eqref{eq:mixed-discretized}.
This coupling, together with the Galerkin coarsening, leads to rather dense LU factors.
In 3D, the fill-in is more severe: 
we observed a 5 times heavier coarse solve for monolithic method, compared to the block-acoustic, for a two-level method applied to a toy-size grid of $64\times 64\times 32$ cells.
Since the fill-in is characterized by super-linear growth, for large 3D grids the LU decomposition constitutes the majority of the overall cost, making the advantage of the block-acoustic approach even bigger.

\subsubsection{Iteration count and computational cost}

Tables \ref{tab:2Dlinear} and \ref{tab:2Dmarmousi} show the iteration count of the block-acoustic multigrid vs. monolithic multigrid. 
The shift for each experiment was chosen by trial and error, to optimize the convergence.
The total computational cost (given in parentheses) is computed by multiplying the iteration number by the corresponding number from Table \ref{tab:nnz} and normalizing\footnote{For brevity, Table \ref{tab:nnz} includes only the grid sizes used for the linear model. 
For other media, we use similarly calculated values that we do not present here.}.

\begin{table} 
\centering
\begin{tabular}{cc|ccc|ccc}
\toprule
  \mc{8}{c}{GMRES(5) iteration count for linear media}\\
\midrule
 & &  \mc{3}{c|}{Block-acoustic multigrid} & \mc{3}{c}{Monolithic multigrid} \\
\midrule
 & & \small 2-level & \small 3-level & \small 4-level & \small 2-level  & \small 3-level & \small 4-level \\
Grid size & $\omega$ & \small $\alpha = 0.1$ &  \small $\alpha=0.2$  & \small $\alpha=0.4$   & \small $\alpha=0.1$   & \small $\alpha=0.4$ & \small $\alpha=0.5 $\\
\midrule
\small $400\times 128$ & \small $3.5\pi$ & \small 31 (0.6) & \small 49 (1.3) & \small 90 (3) & \small 31 (1.1) & \small 79 (2.8) & \small 98 (4.8) \\ 
\small $800\times 256$ & \small $7\pi$ & \small 49 (1.1) & \small 93 (2.5) & \small 171 (5.8) & \small 67 (3) & \small 171 (6.8) & \small 216 (10.8) \\ 
\small $1600\times 512$ & \small $14\pi$ & \small 86 (2.4) & \small 204 (6) & \small 361 (12.5) & \small 161 (9) & \small 414 (18.4) & \small 519 (27.1) \\ 
\bottomrule
 \end{tabular}
\caption{Number of preconditioning cycles needed for convergence with the monolithic  or block-acoustic multigrid, for the 2D elastic Helmholtz equation in linear media. In the monolithic preconditioner, a red-black cell-wise Vanka smoother is used for a $W(1,1)$ cycle, and in the block-acoustic  multigrid, a damped Jacobi smoother is used for a $W(1,2)$ cycle. In parentheses, ten-thousands of $n_{cells}$ FLOPs.
The shift parameter $\alpha$ from \eqref{eq:shift} was chosen to optimize convergence, by an exhaustive search in a resolution of 0.1.
}
\label{tab:2Dlinear}
\end{table}

\begin{table} 
\centering
\begin{tabular}{cc|ccc|ccc}
\toprule
  \mc{8}{c}{GMRES(5) iteration count for Marmousi2 media}\\
\midrule
 & & \mc{3}{c|}{Block-acoustic multigrid} & \mc{3}{c}{Monolithic multigrid} \\
\midrule
 & & \small 2-level & \small 3-level & \small 4-level & \small 2-level  & \small 3-level & \small 4-level \\
Grid size & $\omega$ & \small $\alpha = 0.2$ &  \small $\alpha=0.2$    & \small $\alpha=0.5$  & \small $\alpha=0.1$   & \small $\alpha=0.2$  & \small $\alpha=0.6$ \\
\midrule
\small $544\times 128$ & \small $1.8\pi$ & \small 78 (1.5) & \small 78 (2) & \small 143 (4.7) & \small 30 (1) & \small 54 (1.9)  & \small 152 (7.4) \\
\small $1088\times 240$ & \small $3.6\pi$ & \small 146 (3.2) & \small 143 (3.9) & \small 305 (10.3) & \small 65 (2.9) & \small 124 (4.8) & \small 386 (19.3) \\
\small $2176\times 464$ & \small $7.2\pi$ & \small 445 (12.1) & \small 444 (12.9)  & \small 662 (23) & \small 140 (7.3) & \small 408 (17.6) & \small 813 (42.4) \\
\bottomrule
 \end{tabular}
\caption{Number of preconditioning cycles needed for convergence with the monolithic  or block-acoustic multigrid, for the 2D elastic Helmholtz equation in Marmousi2 media. In the monolithic preconditioner, a red-black cell-wise Vanka smoother is used for a $W(1,1)$ cycle, and in the block-acoustic  multigrid, a damped Jacobi smoother is used for a $W(1,2)$ cycle. In parentheses, ten-thousands of $n_{cells}$ FLOPs. 
The shift parameter $\alpha$ from \eqref{eq:shift} was chosen to optimize convergence, by an exhaustive search in a resolution of 0.1.
}
\label{tab:2Dmarmousi}
\end{table}

Table \ref{tab:2Dlinear} includes the GMRES(5) iteration count and total computational cost for linear media.
 The block-acoustic multigrid achieves a significantly smaller iteration count for 3-level, 4-level and for the larger grids with 2-level. Taking into account the cost per iteration, the block-acoustic multigrid preconditioner accelerates the solution by a factor of 2 and more.

As explained in Subsection \ref{subsec:sl}, full scalability with respect to the grid size cannot be expected for any shifted method.
However, the block-acoustic multigrid yields significantly improved scaling:
the average growth in iterations for the block-acoustic approach varies from $\times1.7$ for the 2-level method to $\times2$ for the 4-level method, while for the monolithic approach the growth factor is about $\times 2.3$, regardless of the levels.

Table \ref{tab:2Dmarmousi} shows the GMRES(5) iteration count and total computational cost for Marmousi2 media.
The iteration count for the 2-level method, as well as the iteration growth factor, is significantly higher for the block-acoustic method, which can be explained by the larger shift required for the method.
For the 3-level method, the iteration count is comparable between the two methods, and for the 4-level method the bock-acoustic multigrid requires less shift and less iterations.
The iteration growth factor for 3- and 4-levels method shows improved scalability: $\times 2.4$ compared to $\times 2.7$ for the 3-level method, and $\times 2.15$ compared to $\times 2.3$ for the 4-level method.
For large enough shift values, the block-acoustic achieves a significantly lower computational cost --- lower by about a factor of 2 for large grids --- and an improved scalability, both for smooth and nonsmooth media.

We note that the monolithic 2-level method has the lowest FLOP count in both Tables \ref{tab:2Dlinear} and \ref{tab:2Dmarmousi}.
However, 2-level methods are not applicable for large 3D cases, since the coarse grid operator is too large to solve exactly. 
Our preconditioner outperforms the monolithic multigrid in larger number of levels, and the required shift seems not to be restrictive, as it stems from the shifted Laplacian multigrid method applied to each block.

\subsubsection{Sensitivity to multigrid components and discretization}
\label{subsubsec:sensitivity}

Throughout the experiments, many choices of multigrid components have been done. 
The choice of smoother, number of smoothing cycles, intergrid operators and other components, might affect the performance of our method, similarly to their influence on the shifted Laplacian multigrid preconditioner.
Although the focus of this paper is \emph{not} the designing of an optimal multigrid framework, we compare here several different components, to justify our previous choices.

In Table \ref{tab:sensitivity} we fix the size of the problem, the media and the number of levels, and vary the intergrid, smoother, number of pre and post relaxations and the shift.
We see that damped Jacobi and Gauss-Seidel perform similarly as smoothers. 
Increasing the total number of relaxations is beneficial to some extent (we did not include here results for $W(2,2)$ cycles since the performance deteriorates compared to $W(1,2)$ cycles). 
We refer to our default intergrid choice, $R_{u_1}$ and $R_{u_2}$ from \eqref{eq:Ru1} and $R_p$ from \eqref{eq:Rp}, as \emph{bilinear} intergrid, and to the choice of bilinear prolongation and lower order restriction, such as $\tilde{R}_{u_1}$ from \eqref{eq:tildeRu1}, as \emph{mixed} intergrid. 
Table \ref{tab:sensitivity} shows that the bilinear intergrid performs significantly better, and require a lower shift.

\begin{table} 
\centering
\begin{tabular}{c|cc|cc|cc|cc}
\toprule
  \mc{9}{c}{GMRES(5) iterations count for different MG parameters}\\
\midrule\
 & \mc{4}{c|}{Damped Jacobi smoother} & \mc{4}{c}{Gauss-Seidel smoother} \\
 \midrule
 & \mc{2}{c|}{\small $W(1,1)$} & \mc{2}{c|}{\small $W(1,2)$} & \mc{2}{c|}{\small $W(1,1)$} & \mc{2}{c}{\small $W(1,2)$} \\
Intergrid & \small $\alpha = 0.4$ & \small $\alpha = 0.5$ & \small $\alpha = 0.4$ & \small $\alpha = 0.5$ & \small $\alpha = 0.4$ & \small $\alpha = 0.5$ & \small $\alpha = 0.4$ & \small $\alpha = 0.5$
\\
\midrule
Bilinear  & 106 & 109 & 90 & 101 & 105 & 110 & 88 & 98 
\\
Mixed & 136 & 134 & 105 & 103 & 136 & 134 & 104 & 100 
\\
\bottomrule
 \end{tabular}
\caption{Number of preconditioning cycles needed for convergence in 2D linear media of size $400\times 128$ and $\omega=3.5\pi$, preconditioned by block-acoustic multigrid with different 4-level cycles.
}
\label{tab:sensitivity}
\end{table}

Discretization is not a proper multigrid component. 
Nevertheless, wider discretization stencils (e.g., 9-points instead of 5-points in 2D) were shown to yield better multigrid convergence for the shifted Laplacian applied to acoustic Helmholtz problems, see \cite{umetani2009multigrid}.
Some high-order finite differences discretizations for elastic wave propagation were suggested in
\cite{virieux1986p, vstekl1998accurate, gu201321}.
These discretizations are designed for the original formulation of the elastic Helmholtz equation, and are therefore not directly applicable to the mixed formulation.
A discretization that optimizes multigrid convergence was recently suggested for elastic Helmholtz equation in mixed formulation in \cite{yovel2024lfa}. The weights of this discretization are tuned by local Fourier analysis (LFA), a predictive tool for convergence of multigrid cycles.
The LFA-tuned discretization \cite{yovel2024lfa} is MAC-based, though different weights are chosen for the finite difference stencil.
The resulting stencil is wider, and offers improvements in truncation error, numerical dispersion and multigrid convergence. 
However, the resulting discretization is not symmetric.
We denote the elastic Helmholtz operator discretized by the LFA-tuned discretization by:
\begin{equation}\label{eq:LFA-tuned}
\mathcal{H}_\beta = \begin{bmatrix}
A_\beta & B^T \\
B_\beta & -C
\end{bmatrix}.
\end{equation}
Note that the gradient is \emph{not} (minus) transpose of the divergence.
To apply the block-acoustic preconditioner to this discretization, we replace $B$ by $B_\beta$ in the derivation \eqref{eq:mixed-distributed} and in the construction of $A_p$ in \eqref{eq:Ap}.

\begin{table} 
\centering
\begin{tabular}{cc|cc|cc}
\toprule
  \mc{6}{c}{GMRES(5) iterations count for SEG/EAGE media}\\
\midrule
\mc{2}{c|}{Discretization} & \mc{2}{c|}{Standard} & \mc{2}{c}{LFA-tuned} \\
 \midrule
Grid size (cells) & $\omega$ & $\quad\mathcal{H}\quad$ & $\quad A_2\quad$ & $\quad\mathcal{H}\quad$ & $\quad A_2\quad$ \\
\midrule
$128\times 256$ & $2.76\pi$ & 34 & 31 & 34 & 30  \\
$256\times 512$ & $5.52\pi$ & 52 & 62 & 55 & 58 \\
$512\times 1024$ & $11.04\pi$ & 107 & 217 & 117 &  204 \\
\bottomrule
 \end{tabular}
\caption{Number of preconditioning cycles needed for convergence for the 2D elastic Helmholtz equation in SEG/EAGE salt media.
We use 3-level $W(1,2)$-cycles with a shift parameter $\alpha=0.1$ and damping parameter 0.8, applied to $\mathcal{H}$ with block-acoustic multigrid or to the block $A_2$ with shifted Laplacian multigrid.} 
\label{tab:disc-sensitivity}
\end{table}

In Table \ref{tab:disc-sensitivity} we compare the solution of the 2D elastic Helmholtz equation in SEG/EAGE salt media, for standard MAC discretization \eqref{eq:mixed-discretized} and for the LFA-tuned discretization \eqref{eq:LFA-tuned}.
To explore the relations between the block-preconditioner and its blocks, we also include results for solving the block $A_2$ by shifted Laplacian multigrid preconditioner. 
The convergence rate is similar for the standard and the LFA-tuned discretization.
In \cite{yovel2024lfa}, the monolithic multigrid solver showed an improvement by $20$--$50\%$ when using re-discretization, and the results here shows smaller or no advantage for this discretization when using Galerkin coarsening.
Yet, there is a slight improvement for the block $A_2$ when using the LFA-tuned discretization, and we suspect that the lack of improvement for the block-acoustic preconditioner applied to $\mathcal{H}$ stems from the non-symmetry of the discretization \eqref{eq:LFA-tuned}, since the commutator ``ignores'' the structure of the gradient.

While Helmholtz problems are sometimes solved on unstructured meshes \cite{babuvska1995generalized, conen2014coarse, kirby2021finite} (especially scattering problems), in this paper we focus on geophysical applications which are typically handled on structured grids.
The analysis above be relevant to unstructured meshes as well, which is beyond the scope of this work.

\subsection{Three-dimensional experiments}\label{subsec:3d}

In this subsection we demonstrate the applicability of our preconditioner, combined with a multigrid solve of each block, to 3D problems.
We preform the experiments on the Overthrust model \cite{aminzadeh19963}, a geophysical 3D model with jumping coefficients. 
This is an acoustic model, and its pressure wave velocity is depicted in Figure \ref{fig:Overthrust}. 
There are several approaches in literature for extending this model to elastic media. 
Motivated by the approach suggested in \cite{brossier2010data},
we extend the originally acoustic model to elastic assuming a constant Poisson ratio $\sigma=0.33$ by defining the shear wave velocity as $V_s = 0.5 V_p$.
The density is defined by a linear stretch of the wave velocity $\rho = 0.25 V_p + 1.2$ such that the density range is physical $\rho\in[1.76,2.7]$ and similar to the 2D elastic model Marmousi2.
Another approach suggested in \cite{scarinci2023robust},
uses the Gardner's law \cite{gardner1974formation} to calculate the density: 
\begin{equation}
\rho = 1.74 \cdot V_p ^\frac{1}{4}
\end{equation}
and after defining a range of shear wave velocities $V_s\in[1.09,3.53]$ we calculate the shear wave velocity by the following linear transformation applied to the pressure wave velocity:
\begin{equation}
V_s = \min(V_s) + \frac{\left(V_p - \min(V_p)\right)\left(\max(V_s)-\min(V_s)\right)}{\max(V_p)-\min(V_p)}.
\end{equation}
It results in a spatially varying Poisson ratio in the range $\sigma\in[0.23,0.39]$.
The model is shallow, and extended similarly to Marmousi2.

\begin{figure}
  \centering
  \includegraphics[width=0.68\textwidth]{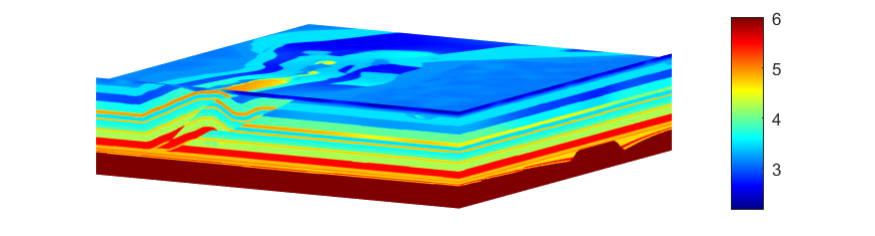}\\
  \caption{The Overthrust model, pressure wave velocity. Velocity units: $km/sec$.}\label{fig:Overthrust}
\end{figure}

The multigrid framework is similar to Subsection \ref{subsec:monolithic_comparison}, with two modifications:
First, we use mixed intergrid as in \eqref{eq:tildeRu1}. 
Second, we take Jacobi $W(2,2)$ cycles with damping parameters of 0.8, 0.8 and 0.2 for the first, second and third level respectively.
These modifications, as we observed by trial and error, improve convergence of the method for the 3D case.

The results are included in Table \ref{tab:3Doverthrust}, demonstrating that the block-acoustic preconditioner enables the solution of the 3D elastic Helmholtz equation in the Overthrust medium with a small number of iterations.
The method performs similarly for the different elastic extensions of the acoustic Overthrust model.
The two versions has similar shear wave velocity ranges and the chosen frequency corresponds to about 10.3-10.5 grid points per wavelength for both.
Moreover, the memory consumption is notably low: 
we manage to solve here a problem with about 46 million unknowns using a laptop with 32 GB memory only.
For the sake of comparison, in \cite{treister2024hybrid} and in \cite{yovel2024lfa} the largest systems to be solved by a 256 GB workstation had about 70 million unknowns, although the LU factors of the coarse grid were \emph{not} computed and saved in any of these works. A domain decomposition coarse grid solver had been applied in \cite{treister2024hybrid}, and the coarse grid in \cite{yovel2024lfa} was solved iteratively with hybrid Kaczmarz relaxation as a preconditioner.

\begin{table} 
\centering
\begin{tabular}{cc|ccc|ccc}
\toprule
  \mc{8}{c}{Iteration count for 3D Overthrust media}\\
\midrule
\mc{2}{c|}{Poisson ratio} & \mc{3}{c|}{$\sigma \equiv 0.33$} & \mc{3}{c}{$\sigma\in[0.23,0.39]$} \\
& & \small 2-level & \small 3-level & \small 4-level & \small 2-level & \small 3-level & \small 4-level  \\
Grid size & $\omega$ &  \small $\alpha=0.1$  & \small $\alpha=0.2$ & \small $\alpha=0.4$ &  \small $\alpha=0.1$  & \small $\alpha=0.2$ & \small $\alpha=0.4$  \\
\midrule
 $128\times128\times56$ & $1.36\pi$ & 24 & 22 & 32 & 26 & 24 & 30 \\
 $192\times192\times72$ & $2.02\pi$ & 29 & 29 & 49 & 30 & 29 & 41 \\
 $256\times256\times96$ & $2.65\pi$ & oom & 38 & 69 & oom & 38 & 53 \\
 $320\times320\times112$ & $3.30\pi$ & oom & 43 & 98 & oom & 46 & 67 \\
\bottomrule
 \end{tabular}
\caption{Number of block-acoustic preconditioned GMRES(5) iterations needed for convergence with block-acoustic multigrid, for the 3D elastic Helmholtz equation in Overthrust media. In each preconditioning step, one $W(2,2)$ cycle with mixed intergrid is applied with damped Jacobi relaxation. The abbreviation ``oom'' stands for an out of memory error on a 32 GB RAM laptop.
}
\label{tab:3Doverthrust}
\end{table}

In this work, we solve large 3D problems without any special adaptation to deal with the coarse grid.
The decoupled block structures enables, besides the improved sparsity of the LU factors, an additional advantage thanks to the natural parallelization.
Since each of the blocks can be solved separately, one can save the 4 pairs of LU factors to the disc, extracting them only when needed, thus reducing by a factor of 4 the total RAM memory consumption of the coarse grid solve. 
The grid sizes in Table \ref{tab:3Doverthrust} were achieved without utilizing disc memory. 
This significant memory save --- which is greater than a factor of four --- can be explained by the fill-in of the LU decomposition.
The saddle-point matrix is not only four times larger than each of its blocks, but also suffers from a severe fill-in due to its sparsity pattern.

\section{Conclusion}\label{sec:conclusion}

In this work we introduced a block-acoustic preconditioner for the elastic Helmholtz equation. 
Building upon the commutation of the underlying operators in the continuous space, our key idea is designing an approximation of the leading block discretized in the pressure's space that will enable a small enough commutator.
It results in a block-triangular preconditioner whose diagonal blocks are acoustic Helmholtz operators, that enables a reduction from elastic to acoustic Helmholtz, and scales well with respect to the Poisson ratio.
Our theoretical results unravel the role of the commutator in the preconditioning of systems with an indefinite leading block.
We demonstrated both analytically and numerically that a small complex shift might be necessary, in highly heterogeneous media, to promise fast convergence. 
However, when problems with smooth media are solved without a complex shift, the block-acoustic preconditioner is scalable with respect to the grid size.

We designed an efficient combination of our preconditioner with shifted Laplacian multigrid applied to each acoustic block, and compared the method to a recent monolithic multigrid preconditioner.
The resulting block-acoustic multigrid  achieves significantly lower computational cost for smooth media.
For challenging geophysical media with jumping coefficients, our approach attains comparable iteration count, yet lower computational cost, compared to the monolithic multigrid.
The block structure of our preconditioner 
enables a huge memory save. 
Overall, the block-acoustic preconditioner makes the problem of elastic wave propagation in large 3D cases applicable using reasonable computational resources.

\section*{Acknowledgements}
The authors are grateful to the late Prof. Howard Elman, for inspiring discussions, for his careful reading of the manuscript, and for helpful comments and pointers to the literature.

\bibliographystyle{siamplain}
\bibliography{Helmholtzbib}
%\bibliography{BlockAcousticSISC_FINAL.bbl}
\end{document}